\documentclass[11pt,a4paper]{amsart}
\usepackage{geometry}
\usepackage[colorlinks, linkcolor=red, anchorcolor=blue, citecolor=blue]{hyperref}
\usepackage{color}             
\usepackage{amsthm,bm} 
\usepackage{enumerate}
\usepackage{amsfonts}
\usepackage{amsmath}
\usepackage{mathrsfs}
\usepackage{amssymb}
\usepackage{amsbsy}
\usepackage{ifthen}
\usepackage[all]{xy}
\usepackage{extarrows}
\usepackage{indentfirst}        
\usepackage{latexsym}        

\usepackage{graphicx}
\usepackage{cases}
\usepackage{pifont}
\usepackage{txfonts}
\usepackage{xcolor}
\usepackage{multirow}
\usepackage{caption, subcaption}

\newcounter{maint}

\numberwithin{equation}{section}

\allowdisplaybreaks[4] 

\begin{document}

\newtheorem{theorem}{Theorem}[section]

\newtheorem{lemma}[theorem]{Lemma}

\newtheorem{corollary}[theorem]{Corollary}
\newtheorem{proposition}[theorem]{Proposition}

\theoremstyle{remark}
\newtheorem{remark}[theorem]{Remark}

\theoremstyle{definition}
\newtheorem{definition}[theorem]{Definition}

\theoremstyle{definition}
\newtheorem{conjecture}[theorem]{Conjecture}

\newtheorem{example}[theorem]{Example}
\newtheorem{problem}[theorem]{Problem}


\def\k{\Bbbk}
\def\I{\mathbb{I}}
\def\ufo{\mathfrak{ufo}}
\newcommand{\Dchaintwo}[4]{
\rule[-3\unitlength]{0pt}{8\unitlength}
\begin{picture}(14,5)(0,3)
\put(2,4){\ifthenelse{\equal{#1}{l}}{\circle*{4}}{\circle{4}}}
\put(4,4){\line(2,0){20}}
\put(26,4){\ifthenelse{\equal{#1}{r}}{\circle*{4}}{\circle{4}}}
\put(2,10){\makebox[0pt]{\scriptsize #2}}
\put(14,8){\makebox[0pt]{\scriptsize #3}}
\put(26,10){\makebox[0pt]{\scriptsize #4}}
\end{picture}}
\title[Finite dimensional Nichols algebras over  Suzuki algebra 
\uppercase\expandafter{\romannumeral2}]
{Finite dimensional Nichols algebras over  Suzuki algebra 
\uppercase\expandafter{\romannumeral2}:
over simple Yetter-Drinfeld modules of $A_{N\,2n+1}^{\mu\lambda}$}

\author[Shi]{Yuxing Shi }
\address{School of Mathematics and Statistics, Jiangxi Normal University,  Nanchang 330022, P. R. China}\email{yxshi@jxnu.edu.cn}

\thanks{
\\
This work was partially supported by
Foundation of Jiangxi Educational Committee (No.12020447)
 and Natural Science Foundation of Jiangxi Normal University(No. 12018937)
}

\begin{abstract}
In this paper, the author gives a complete set of simple Yetter-Drinfeld modules over Suzuki algebra $A_{N\,2n+1}^{\mu\lambda}$ \cite{Suzuki1998} and investigates the Nichols algebras over those irreducible Yetter-Drinfeld modules. 
The finite dimensional Nichols algebras of 
diagonal type are of Cartan type $A_1$, $A_1\times A_1$, $A_2$, Super type 
${\bf A}_{2}(q;\I_2)$ and the Nichols algebra $\ufo(8)$. And the involved finite dimensional Nichols algebras of non-diagonal type are $12$,  $4m$ and $m^2$ dimensional.  
The left three unsolved cases are set as open problems. 
\end{abstract}
\maketitle

\section{Introduction}

Let $\k$ be  an algebraicaly  closed field of characteristic $0$. The  paper is a sequel of 
\cite{Shi2020even} which have started the  classification project of finite dimensional 
Hopf algebras over Suzuki algebra 
$A_{Nn}^{\mu\lambda}$.

In 2004, Menini  et al. studied the quantum lines over $\mathcal{A}_{4m}$ and $\mathcal{B}_{4m}$, which are isomorphic to 
$A_{1m}^{++}$ and $A_{1m}^{+-}$ respectively\cite{MR2037722}. In 2019,  the author \cite{Shi2019} classified  finite dimensional Hopf algebras over the Kac-Paljutkin algebra $A_{12}^{+-} $ and Fantino et al. \cite{Fantino2019} classified finite dimensional Hopf algebras over the dual of dihedral group $\widehat{D}_{2m}$ of order $2m$, with  $m=4a\geq 12$, where $\widehat{D}_{2m}$ is a $2$-cocycle deformation of 
$A_{1\,2a}^{++}$ \cite{MR1800713}. 
In \cite{Shi2020even}, the author gave a complete set of simple Yetter-Drinfeld modules over Suzuki algebra $A_{N\,2n}^{\mu\lambda}$ and investigated the Nichols algebras over those irreducible Yetter-Drinfeld modules. In this paper, we are going to deal with the case $A_{N\,2n+1}^{\mu\lambda}$.

Our classification project over Suzuki algebra is based on the \textit{lifting method} which was introduced by Andruskiewitsch and Schneider \cite{MR1659895}. The lifting method is  a general framework  to classify finite dimensional Hopf algebras with a fixed sub-Hopf algebra as coradical.  
Here let us recall the procedure for the lifting method briefly. Let $H$ be a Hopf algebra whose coradical $H_0$ is a Hopf subalgebra. The associated
graded Hopf algebra of $H$ is isomorphic to $R\#H_0$ where
$R = \oplus_{n\in \mathbb N_0}R(n)$ is a braided Hopf algebra in the category ${}_{H_0}^{H_0}\mathcal{YD}$ of Yetter-Drinfield modules over $H_0$, $\#$ stands for the Radford  \textit{biproduct} or \textit{bosonization} of
$R$ with $H_0$. As explained in \cite{andruskiewitsch2001pointed}, to classify finite-dimensional Hopf algebras $H$ whose coradical is isomorphic to $H_0$ we have to deal with the following questions:
\begin{enumerate}\renewcommand{\theenumi}{\alph{enumi}}\renewcommand{\labelenumi}{(\theenumi)}
  \item\label{que:nichols-fd} Determine all Yetter-Drinfield modules $V$ over $H_0$ such that the Nichols algebra $\mathfrak{B}(V)$ has finite dimension; find an efficient set of relations for $\mathfrak{B}(V)$.
\item\label{que:nichols-R} If $R = \oplus_{n\in \mathbb N_0}R(n)$ is a finite-dimensional Hopf algebra in ${}_{H_0}^{H_0}\mathcal{YD}$ with $V = R(1)$, decide if $R \simeq \mathfrak{B}(V)$. Here $V = R(1)$ is a braided vector space called the \textit{infinitesimal braiding}.
\item\label{que:lifting} Given $V$ as in \eqref{que:nichols-fd}, classify all $H$ such that $\mathrm{gr}\, H \simeq \mathfrak{B}(V)\#H_0$ (lifting).
\end{enumerate}

According to Radford's method \cite[Proposition 2]{radford2003oriented}, we constructed a complete set 
of simple Yetter-Drinfeld modules over $A_{N\,2n+1}^{\mu\lambda}$. There are exactly 
$8N^2$ one-dimensional, $8N^2n(n+1)$ two-dimensional and $8N^2$ $(2n+1)$-dimensional non-isomorphic Yetter-Drinfeld modules over $A_{N\,2n+1}^{\mu\lambda}$, see the Theorem
\ref{YDM}.

The finite dimensional Nichols algebras of diagonal type over simple Yetter-Drinfeld modules of 
$A_{N\,2n+1}^{\mu\lambda}$ can be classified by the following theorem. 

\begin{theorem}\label{MainTheorem}
Let $M$ be a simple Yetter-Drinfeld module over $A_{N\,2n+1}^{\mu\lambda}$. If 
$\mathfrak{B}(M)$ is of diagonal type and $\dim \mathfrak{B}(M)<\infty$, then $\mathfrak{B}(M)$
can be classified as follows. 
\begin{enumerate}
\item Cartan type $A_1$, see Lemmas \ref{NA_A} and \ref{NA_B};
\item Cartan type $A_1\times A_1$, see Lemmas \ref{NA_C}, \ref{NA_D},
          \ref{NA_E}, \ref{NA_F}, \ref{NA_G}, \ref{NA_H} and \ref{NA_I};
\item Cartan type $A_2$, see Lemmas \ref{NA_C}, \ref{NA_D},
          \ref{NA_E}, \ref{NA_F}, \ref{NA_G}, \ref{NA_H} and \ref{NA_I};
\item Super type ${\bf A}_{2}(q;\I_2)$, see Lemmas \ref{NA_C} and \ref{NA_D};
\item The Nichols algebra $\ufo(8)$,  see Lemmas \ref{NA_C} and \ref{NA_D}.
\end{enumerate}
\end{theorem}

There are three unsolved cases which are difficult in general in the paper. 
\begin{problem}
Determine the dimensions of the following Nichols algebras. 
\begin{enumerate}
\item $\mathfrak{B}\left(\mathscr{L}_{k,pq}^s\right)$, $n\geq 2$,  see the section \ref{section_NAL};
\item $\mathfrak{B}(V_{abe})$, $b^2\neq ae\neq 1$, $b^2\neq (ae)^{-1}$, $b\in\Bbb{G}_{m}$ for $m\geq 3$;
\item $\mathfrak{B}\left(\mathscr{N}_{k,pq}^s\right)$,  $n\geq 2$, see the section \ref{section_NAN}.
\end{enumerate}
\end{problem}

The paper is organized as follows. In section 1, we introduce the  background of the paper and summarize our main results. In section 2, we make an introduction for the Suzuki algebra and
 construct all simple representation of $A_{N\,2n+1}^{\mu\lambda}$.
 In section 3, we construct all simple Yetter-Drinfeld modules over $A_{N\,2n+1}^{\mu\lambda}$
 by using Radford's method and we put those Yetter-Drinfeld modules in the appendix.  In section 4, we calculate Nichols algebras over simple Yetter-Drinfeld modules of $A_{N\,2n+1}^{\mu\lambda}$.

\section{The Representations of the Hopf algebra $A_{N\,2n+1}^{\mu\lambda}$}
Suzuki introduced a family of cosemisimple Hopf algebras $A_{Nn}^{\mu\lambda}$ parametrized by integers $N\geq 1$, $n\geq 2$ and $\mu$, $\lambda=\pm 1$, and investigated various properties and structures of them \cite{Suzuki1998}.
Wakui studied the Suzuki algebra $A_{Nn}^{\mu\lambda}$ in perspectives of 
polynomial invariant\cite{Wakui2010a}, braided Morita invariant\cite{Wakui2019} and coribbon structures\cite{Wakui2003}. 
The Hopf algebra  $A_{Nn}^{\mu\lambda}$ is generated by $x_{11}$, $x_{12}$, $x_{21}$, 
$x_{22}$ subject to the relations:
\begin{align*}
&x_{11}^2=x_{22}^2,\quad x_{12}^2=x_{21}^2,\quad \chi _{21}^n=\lambda\chi _{12}^n,
\quad \chi _{11}^n=\chi _{22}^n,\\
&x_{11}^{2N}+\mu x_{12}^{2N}=1,\quad
x_{ij}x_{kl}=0\,\, \text{whenever $i+j+k+l$ is odd}, 
\end{align*}
where we use the following notation for $m\geq 1$, 
$$\chi _{11}^m:=\overbrace{x_{11}x_{22}x_{11}\ldots\ldots }^{\textrm{$m$ }},\quad \chi _{22}^m:=\overbrace{x_{22}x_{11}x_{22}\ldots\ldots }^{\textrm{$m$ }},$$
$$\chi _{12}^m:=\overbrace{x_{12}x_{21}x_{12}\ldots\ldots }^{\textrm{$m$ }},\quad \chi _{21}^m:=\overbrace{x_{21}x_{12}x_{21}\ldots\ldots }^{\textrm{$m$ }}.$$

The Hopf algebra structure of $A_{Nn}^{\mu\lambda}$ is given by
\begin{equation}\label{eq5.3}
\Delta (\chi_{ij}^k)=\chi_{i1}^k\otimes \chi_{1j}^k+\chi_{i2}^k\otimes \chi_{2j}^k,\quad 
\varepsilon(x_{ij})=\delta_{ij}, \quad S(x_{ij})=x_{ji}^{4N-1}, 
\end{equation}
for $k\geq 1$, $i,j=1,2$. 

Let $\overline{i,i+j}=\{i,i+1,i+2,\cdots,i+j\}$ be  an index set. 
Then the basis of $A_{Nn}^{\mu\lambda}$ can be represented by 
\begin{equation}\label{eq5.2}
\left\{x_{11}^s\chi _{22}^t,\ x_{12}^s\chi _{21}^t  \mid
s\in\overline{1,2N}, t\in\overline{0,n-1}
\right\}.
\end{equation}

Thus for $s,t\geq 0$ with $s+t\geq 1$, 
\begin{align*}
\Delta (x_{11}^s\chi _{22}^t)
&=x_{11}^s\chi _{22}^t\otimes x_{11}^s\chi _{22}^t
+x_{12}^s\chi _{21}^t\otimes x_{21}^s\chi _{12}^t,\\ 
\Delta (x_{12}^s\chi _{21}^t)
&=x_{11}^s\chi _{22}^t\otimes x_{12}^s\chi _{21}^t
+x_{12}^s\chi _{21}^t\otimes x_{22}^s\chi _{11}^t. 
\end{align*}

The cosemisimple Hopf algebra $A_{Nn}^{\mu\lambda}$ is decomposed to the direct sum of simple subcoalgebras such as $A_{Nn}^{\mu\lambda}
=\bigoplus_{g\in G}\k g\oplus\bigoplus_{\substack{0\leq s\leq N-1\\ 1\leq t\leq n-1}}C_{st}$
\cite[Theorem 3.1]{Suzuki1998}\cite[Proposition 5.5]{Wakui2010a}, 
where 
\begin{align*}
G&=\left\{x_{11}^{2s}\pm x_{12}^{2s}, x_{11}^{2s+1}\chi_{22}^{n-1}\pm 
        \sqrt{\lambda}x_{12}^{2s+1}\chi_{21}^{n-1}\mid s\in\overline{1,N}\right\},\\
C_{st}&=\k x_{11}^{2s}\chi_{11}^t+\k x_{12}^{2s}\chi_{12}^t+
              \k x_{11}^{2s}\chi_{22}^t+\k x_{12}^{2s}\chi_{21}^t,\quad
              s\in\overline{1,N}, t\in\overline{1,n-1}. 
\end{align*}
The set $\left\{\k g\mid g\in G\right\}\cup \left\{\k x_{11}^{2s}\chi_{11}^t
+\k x_{12}^{2s}\chi_{21}^t \mid s\in\overline{1,N}, t\in\overline{1,n-1} 
\right\}$ 
is a full set of non-isomorphic simple left $A_{Nn}^{\mu\lambda}$-comodules, where 
the coactions of the comodules listed above are given by the coproduct $\Delta$. Denote 
the comodule $\k x_{11}^{2s}\chi_{11}^t
+\k x_{12}^{2s}\chi_{21}^t $ by $\Lambda_{st}$. That is to say the comodule 
$\Lambda_{st}=\k w_1+\k w_2$
is defined as 
\begin{align*}
\rho\left(w_1\right)
=  x_{11}^{2s}\chi_{11}^t\otimes w_1
     +x_{12}^{2s}\chi_{12}^t\otimes w_2,\quad 
\rho\left(w_2\right)
=  x_{11}^{2s}\chi_{22}^t\otimes w_2
     +x_{12}^{2s}\chi_{21}^t\otimes w_1 .
\end{align*}

\begin{proposition}
Let   $\omega$ be a primitive $4(2n+1)N$-th root of unity, and 
\[
\tilde{\mu}=\left\{\begin{array}{rl}
1,&\mu=1,\\
\omega^{2n+1}, &\mu=-1, 
\end{array}\right.
\quad
\bar{\mu}=\left\{\begin{array}{rl}
1,&\mu=1,\\
\omega^{2(2n+1)}, &\mu=-1.
\end{array}\right.
\]
Then a full set of non-isomorphic simple left $A_{N\,2n+1}^{\mu\lambda}$-modules is given by
\begin{enumerate}
\item $V_k$, $k\in\overline{0,2N-1}$. The action of $A_{N\,2n+1}^{\mu\lambda}$ on $V_{k}$ is given by 
\[
x_{12}\mapsto 0,\quad x_{21}\mapsto 0,\quad 
 x_{11}\mapsto \omega^{2k(2n+1)},\quad  x_{22}\mapsto \omega^{2k(2n+1)};
 \]
 
\item $V_k^\prime$, $k\in\overline{0,2N-1}$. The action of $A_{N\,2n+1}^{\mu\lambda}$ on $V_{k}^\prime$ is given by 
\[
x_{11}\mapsto 0,\quad x_{22}\mapsto 0,\quad 
x_{12}\mapsto \omega^{2k(2n+1)}\tilde{\mu},\quad  
x_{21}\mapsto \omega^{2k(2n+1)}\tilde{\mu}\lambda;
\]

\item $V_{jk}$, $k\in\overline{0,N-1}$, 
            $\frac j2\in \overline{1,n}$.
           The action of $A_{N\,2n+1}^{\mu\lambda}$ on the row vector $(v_{1}, v_{2})$ 
           for $V_{jk}=\k v_{1}\oplus \k v_{2}$ is given by 
           \[
           x_{11}\mapsto \begin{pmatrix}0&\omega^{4k(2n+1)-2jN}\\
           \omega^{2jN}&0
           \end{pmatrix},\quad x_{12},x_{21}\mapsto 0,\quad 
            x_{22}\mapsto \begin{pmatrix}0&\omega^{4k(2n+1)}\\
           1&0
          \end{pmatrix};
           \]

\item  $V_{jk}^\prime$, $k\in\overline{0,N-1}$, 
            $\frac j2\in \overline{1,n}$. 
The action of $A_{N\,2n+1}^{\mu\lambda}$ on the row vector $(v_{1}^\prime, v_{2}^\prime)$ 
 for $V_{jk}^\prime=\k v_{1}^\prime\oplus \k v_{2}^\prime$  is given by 
 \[
 x_{21}\mapsto \begin{pmatrix}0& \lambda\bar{\mu}\omega^{4k(2n+1)-2jN}\\
 \lambda\omega^{2jN}&0
 \end{pmatrix},\quad x_{11},x_{22}\mapsto 0,\quad 
  x_{12}\mapsto \begin{pmatrix}0&\bar{\mu}\omega^{4k(2n+1)}\\
 1&0
 \end{pmatrix}.
 \]
\end{enumerate}
\end{proposition}
\begin{remark}
 We left the proof to the readers since it's easy and tedious. 
\end{remark}
\section{Yetter-Drinfeld modules over $A_{N\,2n+1}^{\mu\lambda}$}
Similarly according to Radford's method \cite[Proposition 2]{radford2003oriented}, any simple left Yetter-Drinfeld module over  a Hopf algebra $H$ could be constructed by the submodule of  tensor product of a left  module $V$ of $H$ and
$H$ itself, where the  module and comodule structures are given by :
\begin{align}
h\cdot (\ell\boxtimes g)
&=(h_{(2)}\cdot \ell)\boxtimes h_{(1)}gS(h_{(3)}),\label{eq:action}\\
\rho(\ell\boxtimes h)
&=h_{(1)}\otimes (\ell\boxtimes h_{(2)}) , \forall h, g\in H, \ell\in V.
\label{eq:coaction}
\end{align}
Here we use $\boxtimes$ instead of $\otimes$ to avoid confusion by using too many symbols of the tensor product.  we construct all simple left Yetter-Drinfeld modules over $A_{N\,2n+1}^{\mu\lambda}$ in this way and put them in the appendix without proof since it's tedious verification with the definition of Yetter-Drinfeld modules. 
Firstly, it's easy to see that there are $8N^2$ pairwise non-isomorphic simple Yetter-Drinfeld modules of one-dimension from the Table \ref{YDMod1}. Secondly, let us take a closer look at those Yetter-Drinfeld 
modules which are isomorphic as $A_{N\,2n+1}^{\mu\lambda}$-modules and $A_{N\,2n+1}^{\mu\lambda}$-comodules, but the modules isomorphism and comodules isomorphism are  incompatible. For example: 
\begin{enumerate}[I]
\item Yetter-Drinfeld modules in the part (3), (4) of  the list are 
  non-isomorphic since 
         \begin{align*}
         \mathscr{C}_{jk,p}^{st}&=\k w_1\oplus \k w_2\simeq V_{jk}, 
         &\mathscr{C}_{jk,p}^{st}&=\k w_2\oplus \k \left((-1)^p\omega^{2k(2n+1)-jN}w_1\right)
         \simeq \Lambda_{s\, 2t+2},\\
         \mathscr{D}_{jk,p}^{st}&=\k w_1\oplus \k w_2\simeq V_{jk}, 
         &\mathscr{D}_{jk,p}^{st}&=\k w_1\oplus \k \left((-1)^p\omega^{jN-2k(2n+1)}w_2\right)
         \simeq \Lambda_{s\, 2t+2}.
         \end{align*}
\item Yetter-Drinfeld modules in the part (8), (9) of  the list are 
non-isomorphic since 
         \begin{align*}
         \mathscr{H}_{jk,p}^{st}&=\k w_1\oplus \k w_2\simeq V_{jk}^\prime, 
         &\mathscr{H}_{jk,p}^{st}&=\k w_2\oplus 
         \k \left((-1)^p\sqrt{\lambda\bar{\mu}}\omega^{2k(2n+1)-jN}w_1\right)
         \simeq \Lambda_{s\, 2t+1},\\
         \mathscr{I}_{jk,p}^{st}&=\k w_1\oplus \k w_2\simeq V_{jk}^\prime, 
         &\mathscr{I}_{jk,p}^{st}&=\k w_1\oplus 
         \k \left(\frac{(-1)^p}{\sqrt{\lambda\bar{\mu}}}\omega^{jN-2k(2n+1)}w_2\right)
         \simeq \Lambda_{s\, 2t+1}.
         \end{align*}
\end{enumerate}
Now we can see that there are $8N^2n(n+1)$ pairwise non-isomorphic simple Yetter-Drinfeld modules of two-dimension from the Table \ref{YDMod1}. 
While we break the Yetter-Drinfeld module $M\boxtimes A_{N\,2n+1}^{\mu\lambda}$ into small 
Yetter-Drinfeld modules for any simple left $A_{N\,2n+1}^{\mu\lambda}$-module $M$, there are four 
class Yetter-Drinfeld modules in total whose dimensions are greater than two. And they have the relations 
$
\mathscr{K}_{k,p}^s\simeq \mathscr{L}_{k,0p}^s$,
$\mathscr{M}_{k,p}^s\simeq \mathscr{N}_{k,0p}^s
$ as Yetter-Drinfeld modules. Since 
\[
8N^2\cdot 1^2+8N^2n(n+1)\cdot 2^2+8N^2\cdot (2n+1)^2=\left[4N(2n+1)\right]^2,
\]
the list is a complete set of simple Yetter-Drinfeld modules.

\begin{theorem}\label{YDM}
A complete set of simple Yetter-Drinfeld modules over Suzuki algebra $A_{N\,2n+1}^{\lambda\mu}$
is given as follows. 
\begin{enumerate}
\item There are $8N^2$ pairwise  non-isomorphic Yetter-Drinfeld modules of one dimenion:
\begin{enumerate}
\item $\mathscr{A}_{k,p}^s$, $s\in\overline{1,N}$, $k\in\overline{0,2N-1}$, $p\in\Bbb{Z}_2$;
\item $\mathscr{B}_{k,p}^s$, $s\in\overline{1,N}$, $k\in\overline{0,2N-1}$, $p\in\Bbb{Z}_2$.
\end{enumerate}
\item There are $8N^2n(n+1)$ pairwise  non-isomorphic Yetter-Drinfeld modules of two
dimension:
\begin{enumerate}
\item $\mathscr{C}_{jk,p}^{st}$, $s\in\overline{1,N}$, $t\in\overline{0,n-1}$, 
    $\frac j2\in\overline{1,n}$, 
    $k\in\overline{0,N-1}$, $p\in\Bbb{Z}_2$;
\item $\mathscr{D}_{jk,p}^{st}$, $s\in\overline{1,N}$, $t\in\overline{0,n-1}$, 
    $\frac j2\in\overline{1,n}$, 
    $k\in\overline{0,N-1}$, $p\in\Bbb{Z}_2$;
\item $\mathscr{E}_{jk,p}^{s}$, $s\in\overline{1,N}$, 
    $\frac j2\in\overline{1,n}$, 
    $k\in\overline{0,N-1}$, $p\in\Bbb{Z}_2$;
\item $\mathscr{F}_{k,p}^{st}$, $s\in\overline{1,N}$, $t\in\overline{0,n-1}$, 
    $k\in\overline{0,2N-1}$, $p\in\Bbb{Z}_2$;
\item $\mathscr{G}_{k,p}^{st}$, $s\in\overline{1,N}$, $t\in\overline{0,n-1}$, 
    $k\in\overline{0,2N-1}$, $p\in\Bbb{Z}_2$;
\item $\mathscr{H}_{jk,p}^{st}$, $s\in\overline{1,N}$, $t\in\overline{0,n-1}$, 
    $\frac j2\in\overline{1,n}$, 
    $k\in\overline{0,N-1}$, $p\in\Bbb{Z}_2$;
\item $\mathscr{I}_{jk,p}^{st}$, $s\in\overline{1,N}$, $t\in\overline{0,n-1}$, 
    $\frac j2\in\overline{1,n}$, 
    $k\in\overline{0,N-1}$, $p\in\Bbb{Z}_2$;
\item $\mathscr{I}_{jk,p}^{st}$, $s\in\overline{1,N}$, $t=n$, 
    $\frac j2\in\overline{1,n}$, 
    $k\in\overline{0,N-1}$, $p\in\Bbb{Z}_2$.
\end{enumerate}
\item There are $8N^2$ pairwise  non-isomorphic Yetter-Drinfeld modules of $2n+1$
dimension:
\begin{enumerate}
\item $\mathscr{L}_{k,pq}^s$, $s\in\overline{1,N}$, 
    $k\in\overline{0,N-1}$, $p,q\in\Bbb{Z}_2$;
\item $\mathscr{N}_{k,pq}^s$, $s\in\overline{1,N}$, 
    $k\in\overline{0,N-1}$, $p,q\in\Bbb{Z}_2$.
\end{enumerate}
\end{enumerate}
\end{theorem}
\begin{remark}
As for the description of those simple Yetter-Drinfeld modules, please see the Appendix. 
\end{remark}

\begin{table}
\newcommand{\tabincell}[2]{\begin{tabular}{@{}#1@{}}#2\end{tabular}}
\begin{tabular}{|c|c|c|c|c|}
\hline
& dim &  parameters &  mod   & comod\\\hline
$\mathscr{A}_{k,p}^s$ 
& $1$
&\tabincell{c}{\vspace{1mm} $s\in\overline{1,N}$, $p\in\Bbb{Z}_2$\\ $k\in\overline{0,2N-1}$\vspace{1mm}} 
&  $V_k$  
&\tabincell{c}{$\k g_s^+, p=0$\\ $\k g_s^-, p=1$}
\\\hline
$\mathscr{B}_{k,p}^s$ 
& $1$
&\tabincell{c}{\vspace{1mm} $s\in\overline{1,N}$, $p\in\Bbb{Z}_2$\\ $k\in\overline{0,2N-1}$\vspace{1mm}} 
&  $V_k^\prime$  
&\tabincell{c}{$\k h_s^+, p=0$\\ $\k h_s^-, p=1$}
\\\hline
\tabincell{c}{$\mathscr{C}_{jk,p}^{st}$}
& $2$
& \tabincell{c}{$s\in\overline{1,N}$,  $t\in\overline{0,n-1}$,\\
    $\frac j2\in\overline{1,n}$, $p\in\Bbb{Z}_2$,\\$k\in\overline{0,N-1}$}
& $V_{jk}$ 
& $\Lambda_{s\,2t+2}$
\\\hline
\tabincell{c}{$\mathscr{D}_{jk,p}^{st}$}
& $2$
& \tabincell{c}{$s\in\overline{1,N}$,  $t\in\overline{0,n-1}$,\\
    $\frac j2\in\overline{1,n}$, $p\in\Bbb{Z}_2$,\\$k\in\overline{0,N-1}$}
& $V_{jk}$ 
& $\Lambda_{s\,2t+2}$
\\\hline
$\mathscr{E}_{jk,p}^s$
& $2$
& \tabincell{c}{$s\in\overline{1,N}$,  \\
    $\frac j2\in\overline{1,n}$, $p\in\Bbb{Z}_2$,\\$k\in\overline{0,N-1}$}
& $V_{jk}$ 
&\tabincell{c}{$\k g_s^+\oplus \k g_s^-$}
\\\hline
\tabincell{c}{$\mathscr{F}_{k,p}^{st}$}
& $2$
& \tabincell{c}{$s\in\overline{1,N}$, $t\in\overline{0,n-1}$,\\ $k\in\overline{0,2N-1}$, $p\in\Bbb{Z}_2$} 
&\tabincell{c}{$V_k\oplus V_d$,\\ $k+N\equiv d\,\mathrm{mod}\, 2N$}    
&$\Lambda_{s\,2t+2}$
\\\hline
$\mathscr{G}_{k,p}^{st}$
& $2$
&\tabincell{c}{$s\in\overline{1,N}$, $t\in\overline{0,n-1}$,\\ $k\in\overline{0,2N-1}$, $p\in\Bbb{Z}_2$}
& \tabincell{c}{$V_k^\prime\oplus V_d^\prime$,\\$k+N\equiv d\,\mathrm{mod}\, 2N$}
& $\Lambda_{s\,2t+1}$
\\\hline
$\mathscr{H}_{jk,p}^{st}$
& $2$
& \tabincell{c}{$s\in\overline{1,N}$,  $t\in\overline{0,n-1}$,\\
    $\frac j2\in\overline{1,n}$, $p\in\Bbb{Z}_2$,\\$k\in\overline{0,N-1}$}
& $V_{jk}^\prime$
& $\Lambda_{s\,2t+1}$
\\\hline
 \multirow{2}{*}{$\mathscr{I}_{jk,p}^{st}$}
& $2$
& \tabincell{c}{$s\in\overline{1,N}$,  $t\in\overline{0,n-1}$,\\
    $\frac j2\in\overline{1,n}$, $p\in\Bbb{Z}_2$,\\$k\in\overline{0,N-1}$}
& $V_{jk}^\prime$
& $\Lambda_{s\,2t+1}$
\\\cline{2-5}
& $2$
& \tabincell{c}{$s\in\overline{1,N}$,  $ t=n$,\\
    $\frac j2\in\overline{1,n}$, $p\in\Bbb{Z}_2$,\\$k\in\overline{0,N-1}$}
& $V_{jk}^\prime$
&\tabincell{c}{$\k h_s^+\oplus \k h_s^-$}
\\\hline
\end{tabular}
\vspace{3ex}
\caption{Simple Yetter-Drinfeld modules of one or two dimension. Here $g_s^{\pm}=x_{11}^{2s}\pm x_{12}^{2s}$, $h_s^{\pm}=x_{11}^{2s+1}\chi_{22}^{2n}
\pm\sqrt{\lambda} x_{12}^{2s+1}\chi_{21}^{2n}$.}
\label{YDMod1}
\end{table}

\section{Nichols algebras over $A_{N\,2n+1}^{\mu\lambda}$}
In this section, we investigate  Nichols algebras over  simple Yetter-Drinfeld modules
of $A_{N\,2n+1}^{\mu\lambda}$. So the Yetter-Drinfeld modules discussed in the section 
are those listed in the Theorem \ref{YDM}.
For the knowledge about Nichols algebras, please refer to 
\cite{andruskiewitsch2001pointed} \cite{Andruskiewitsch2017a} \cite{Andruskiewitsch2017}.

\subsection{Nichols algebras of diagonal type}
Let $V=\bigoplus_{i\in I}\k v_i$ be a vector space with a braiding $c(v_i\otimes v_j)=q_{ij}v_j\otimes v_i$, 
then the Nichols algebra $\mathfrak{B}(V)$ is  of diagonal type. 
Our results in this section heavily rely on Heckenberger's classification work  \cite{heckenberger2009classification}. To keep the article concise, we don't repeat this in the following proofs.  For more details about Nichols algebras of diagonal type, please consult \cite{Andruskiewitsch2017}. 

\begin{lemma}\label{NA_A}
Let $s\in\overline{1,N}$, $k\in\overline{0,2N-1}$, then 
\[
\dim\mathfrak{B}\left(\mathscr{A}_{k,p}^s\right)
=\left\{\begin{array}{ll} \infty, & N\mid ks,\\
   \frac N{(d,N)}, & ks\equiv d\mod N, d\in\overline{1,N-1}.
   \end{array}\right.
\]
\end{lemma}
\begin{remark}
When $N=1$, then $\dim\mathfrak{B}\left(\mathscr{A}_{k,p}^s\right)=\infty$.\\
When $N=2$, then 
$
\dim\mathfrak{B}\left(\mathscr{A}_{k,p}^s\right)
=\left\{\begin{array}{ll} 
   2, & s=1, k=1\,\text{or}\, 3,\\
   \infty, & otherwise.
   \end{array}\right.
$\\
When $N$ is a prime, then
$
\dim\mathfrak{B}\left(\mathscr{A}_{k,p}^s\right)
=\left\{\begin{array}{ll} 
   N, & s, k\nequiv 0\mod N,\\
   \infty, & otherwise.
   \end{array}\right.
$
\end{remark}
\begin{proof}
$c(w\otimes w)=\omega^{4ks(2n+1)}w\otimes w$.
\end{proof}

\begin{lemma}\label{NA_B} 
Let $s\in\overline{1,N}$, $k\in\overline{0,2N-1}$, $p\in\Bbb{Z}_2$,
\begin{enumerate}
\item when $\lambda=\mu=1$, then 
\begin{align*}
\dim\mathfrak{B}\left(\mathscr{B}_{k,p}^s\right)
&=\left\{\begin{array}{ll}
      \infty, & 2N\mid [pN+k(2s+1+2n)],\\
      \frac{2N}{(2N,d)}, & pN+k(2s+1+2n)\equiv d\mod 2N, d\in\overline{1,2N-1};
      \end{array}\right.
\end{align*}
\item when $\lambda=-1$, $\mu=1$, then
\begin{align*}
\dim\mathfrak{B}\left(\mathscr{B}_{k,p}^s\right)
&=\left\{\begin{array}{ll}
      \infty, & 4N\mid [N(2p+2n+1)+2k(2s+1+2n)],\\
      \frac{4N}{(4N,d)}, & [N(2p+2n+1)+2k(2s+1+2n)]\equiv d\mod 4N, 
      \end{array}\right.
\end{align*}
where $d\in\overline{1,4N-1}$;
\item when $\lambda=1$, $\mu=-1$, then
\begin{align*}
\dim\mathfrak{B}\left(\mathscr{B}_{k,p}^s\right)
&=\left\{\begin{array}{ll}
      \infty, & 4N\mid [2Np+(2s+1+2n)(2k+1)],\\
      \frac{4N}{(4N,d)}, & [2Np+(2s+1+2n)(2k+1)]\equiv d\,\,\mathrm{mod}\,\, 4N, d\in\overline{1,4N-1};
      \end{array}\right.
\end{align*}
\item when $\lambda=\mu=-1$, then 
\begin{align*}
\dim\mathfrak{B}\left(\mathscr{B}_{k,p}^s\right)
&=\left\{\begin{array}{ll}
      \infty, & 4N\mid [N(2p+2n+1)+(2s+1+2n)(2k+1)],\\
      \frac{4N}{(4N,d)}, & [N(2p+2n+1)+(2s+1+2n)(2k+1)]\equiv d\mod 4N, 
      \end{array}\right.
\end{align*}
where $d\in\overline{1,4N-1}$.
\end{enumerate}
\end{lemma}
\begin{remark}
When $N=1$, then
\begin{align*}
\dim\mathfrak{B}\left(\mathscr{B}_{k,p}^s\right)
&=\left\{\begin{array}{ll}
2, &\mu= \lambda=1, s=1, (k,p)=(0,1)\,\text{or}\, (1,0);\\\vspace{1mm}
4, &\mu=1, \lambda=-1, s=1, k\in\overline{0,1}, p\in\Bbb{Z}_2;\\\vspace{1mm}
4, &\mu=-1, \lambda=1, s=1, k\in\overline{0,1}, p\in\Bbb{Z}_2;\\\vspace{1mm}
2, &\mu=\lambda=-1, s=1, (k,p)=(0,1)\,\text{or}\, (1,0);\\\vspace{1mm}
\infty, &otherwise. \\
      \end{array}\right.
\end{align*}
\end{remark}
\begin{proof}
The braiding of $\mathfrak{B}\left(\mathscr{B}_{k,p}^s\right)$ is given by 
\begin{align*}
c(w\otimes w)
&=(-1)^p\sqrt{\lambda}\lambda^n\tilde{\mu}^{2s+1+2n}
     \omega^{2k(2n+1)(2s+1+2n)} w\otimes w\\
&=\left\{\begin{array}{ll}
     \omega^{2(2n+1)[pN+k(2s+1+2n)]} w\otimes w, & \lambda=\mu=1,\\
     \omega^{(2n+1)[N(2p+2n+1)+2k(2s+1+2n)]}w\otimes w, & \lambda=-1, \mu=1,\\
     \omega^{(2n+1)[2Np+(2s+1+2n)(2k+1)]}w\otimes w, & \lambda=1, \mu=-1,\\
     \omega^{(2n+1)[N(2p+1+2n)+(2s+1+2n)(2k+1)]}w\otimes w, & \lambda=\mu=-1.
     \end{array}\right.
\end{align*}
\end{proof}

\begin{lemma}\label{NA_C}
Denote 
         \[\alpha=2Nj(t+1)+4k(2n+1)(t+1+s),\quad 
         \beta=-4Nj(t+1)+8k(2n+1)(t+1+s),\]
then 
$\dim\mathfrak{B}\left(\mathscr{C}_{jk,p}^{st}\right)<\infty$
iff one of the following conditions holds. 
\begin{enumerate}
\item  $4N(2n+1)\nmid \alpha$, $4N(2n+1)\mid \beta$, Cartan type $A_1\times A_1$;
\item  $4N(2n+1)\nmid \alpha$, $4N(2n+1)\mid (\alpha+\beta)$, Cartan type $A_2$;
\item $\alpha\equiv 2N(2n+1)\mod 4N(2n+1)$, $\beta\nequiv 0\,\,\text{and}\,\,2N(2n+1)\mod 
          4N(2n+1)$, Super type ${\bf A}_{2}(q;\I_2)$;
\item $\alpha-2\beta\equiv 6\beta\equiv 2N(2n+1)\mod 4N(2n+1)$, $4\beta\nequiv 0\mod 4N(2n+1)$.
The Nichols algebras $\ufo(8)$, see \cite[Page 209]{Andruskiewitsch2017}.
\end{enumerate}
\end{lemma}
\begin{proof}
The braiding of $\mathfrak{B}\left(\mathscr{C}_{jk,p}^{st}\right)$ is given by
\begin{align*}
c(w_1\otimes w_1)
&=\omega^{[2jN+4k(2n+1)](t+1)+4k(2n+1)s}w_1\otimes w_1\\
c(w_1\otimes w_2)
&=\omega^{[4k(2n+1)-2jN](t+1)+4k(2n+1)s}w_2\otimes w_1\\
c(w_2\otimes w_1)
&=\omega^{[4k(2n+1)-2jN](t+1)+4k(2n+1)s}w_1\otimes w_2\\
c(w_2\otimes w_2)
&=\omega^{[4k(2n+1)+2jN](t+1)+4k(2n+1)s}w_2\otimes w_2.
\end{align*}
\end{proof}

\begin{lemma}\label{NA_D}
Denote 
         \[\alpha=-2Nj(t+1)+4k(2n+1)(t+1+s),\quad
         \beta=4Nj(t+1)+8k(2n+1)(t+1+s),\]
then 
$
\dim\mathfrak{B}\left(\mathscr{D}_{jk,p}^{st}\right)<\infty
$
iff one of the following conditions holds.
\begin{enumerate}
\item  $4N(2n+1)\nmid \alpha$, $4N(2n+1)\mid \beta$, Cartan type $A_1\times A_1$;
\item  $4N(2n+1)\nmid \alpha$, $4N(2n+1)\mid (\alpha+\beta)$, Cartan type $A_2$;
\item $\alpha\equiv 2N(2n+1)\mod 4N(2n+1)$, $\beta\nequiv 0\,\,\text{and}\,\,2N(2n+1)\mod 
          4N(2n+1)$, Super type ${\bf A}_{2}(q;\I_2)$;
\item $\alpha-2\beta\equiv 6\beta\equiv 2N(2n+1)\mod 4N(2n+1)$, $4\beta\nequiv 0\mod 4N(2n+1)$.
The Nichols algebras $\ufo(8)$, see \cite[Page 209]{Andruskiewitsch2017}.
\end{enumerate}
\end{lemma}
\begin{proof}
The braiding of $\mathfrak{B}\left(\mathscr{D}_{jk,p}^{st}\right)$ is given by 
\begin{align*}
c(w_1\otimes w_1)
&=\omega^{[4k(2n+1)-2jN](t+1)+4k(2n+1)s}w_1\otimes w_1,\\
c(w_1\otimes w_2)
&=\omega^{[4k(2n+1)+2jN](t+1)+4k(2n+1)s}w_2\otimes w_1,\\
c(w_2\otimes w_1)
&=\omega^{[4k(2n+1)+2jN](t+1)+4k(2n+1)s}w_1\otimes w_2,\\
c(w_2\otimes w_2)
&=\omega^{[4k(2n+1)-2jN](t+1)+4k(2n+1)s}w_2\otimes w_2.
\end{align*}
\end{proof}

\begin{lemma}\label{NA_E}
Denote $q=\omega^{4ks(2n+1)}$, then $\mathfrak{B}\left(\mathscr{E}_{jk,p}^{s}\right)$ is finite dimensional iff one of the following conditions holds. 
\begin{enumerate}
\item  $q=-1$, Cartan type $A_1\times A_1$;
\item  $q^3=1\neq q$, Cartan type $A_2$.
\end{enumerate} 
\end{lemma}
\begin{proof}
The braiding of $\mathfrak{B}\left(\mathscr{E}_{jk,p}^{s}\right)$ is given by
\begin{align*}
c(w_1\otimes w_1)
&=x_{11}^{2s}\cdot w_1\otimes w_1
=\omega^{4k(2n+1)s}w_1\otimes w_1\\
c(w_1\otimes w_2)
&=x_{11}^{2s}\cdot w_2\otimes w_1
=\omega^{4k(2n+1)s}w_2\otimes w_1\\
c(w_2\otimes w_1)
&=x_{11}^{2s}\cdot w_1\otimes w_2
=\omega^{4k(2n+1)s}w_1\otimes w_2\\
c(w_2\otimes w_2)
&=x_{11}^{2s}\cdot w_2\otimes w_2
=\omega^{4k(2n+1)s}w_2\otimes w_2
\end{align*}
\end{proof}

\begin{lemma}\label{NA_F}
Denote $q=\omega^{2k(2n+1)(2s+2t+1)}$, then 
\[
\dim \mathfrak{B}\left(\mathscr{F}_{k,p}^{st}\right)
=\left\{\begin{array}{ll}
4, & q=-1, (\text{Cartan type $A_1\times A_1$}),\\
27, & q^3=1\neq q, (\text{Cartan type $A_2$}),\\
\infty, & \text{otherwise}.
\end{array}\right.
\]
\end{lemma}
\begin{proof}
The braiding of $\mathfrak{B}\left(\mathscr{F}_{k,p}^{st}\right)$ is given by
\begin{align*}
c(w_1\otimes w_1)
&=x_{11}^{2s+1}\chi_{22}^{2t+1}\cdot w_1\otimes w_1
=\omega^{2k(2n+1)(2s+2t+1)}w_1\otimes w_1\\
c(w_1\otimes w_2)
&=x_{11}^{2s+1}\chi_{22}^{2t+1}\cdot w_2\otimes w_1
=\omega^{2k(2n+1)(2s+2t+1)}w_2\otimes w_1\\
c(w_2\otimes w_1)
&=x_{11}^{2s}\chi_{22}^{2t+2}\cdot w_1\otimes w_2
=\omega^{2k(2n+1)(2s+2t+1)}w_1\otimes w_2\\
c(w_2\otimes w_2)
&=x_{11}^{2s}\chi_{22}^{2t+2}\cdot w_2\otimes w_2
=\omega^{2k(2n+1)(2s+2t+1)}w_2\otimes w_2.
\end{align*}
\end{proof}

\subsection{Nichols algebra of type $V_{abe}$}\label{typeVabe}
Let $V_{abe}=\k v_1\otimes \k v_2$ be a vector space with a braiding given by 
\begin{align*}
c(v_1\otimes v_1)&=a v_2\otimes v_2,\quad 
&c(v_1\otimes v_2)&=b   v_1\otimes v_2,\\
c(v_2\otimes v_1)&=b v_2\otimes v_1,\quad 
&c(v_2\otimes v_2)&=e   v_1\otimes v_1,
\end{align*}
then the braided vector space $(V_{abe}, c)$ is of type $V_{abe}$. The braided vector space $V_{abe}$ is isomorphic to  $V_{ae\,b\,1}$ via $v_1\mapsto \sqrt{e}v_1$, $v_2\mapsto v_2$. When $ae=b^2$, then 
$V_{abe}$ is of diagonal type and
\[
\dim \mathfrak{B}(V_{abe})
=\left\{\begin{array}{ll}
4, & b=-1, (\mathfrak{B}(V_{abe}) \,\,\,\text{is of Cartan type $A_1\times A_1$}),\\
27, & b^3=1\neq b, (\mathfrak{B}(V_{abe})\,\,\, \text{is of Cartan type $A_2$}),\\
\infty, & \text{otherwise}.
\end{array}\right.
\]

When $ae\neq b^2$, according to 
 \cite[section 3.7]{Andruskiewitsch2018} and  \cite{Shi2020even}, we have
\begin{align}\label{fomulaeVabe}
\dim\mathfrak{B}(V_{abe})
=\left\{\begin{array}{ll}
4m, &b=-1, ae\in\Bbb{G}_m,\\
m^2, &ae=1, b\in\Bbb{G}_m\,\,\text{for}\,\,m\geq 2,\\
\infty, & b^2=(ae)^{-1},  b\in\Bbb{G}_{m}\,\, \text{for}\,\, m\geq 5,\\
\infty,   & b\notin \Bbb{G}_m \,\,\text{for}\,\,m\geq 2,\\
\text{unknown}, & otherwise, 
\end{array}\right.
\end{align}
where $\Bbb{G}_m$ denotes the set of $m$-th primitive roots of unity. 

\begin{lemma}\label{NA_G}
Denote 
$q=(-1)^p\lambda^{t+\frac32}\tilde{\mu}^{2s+2t+1}\omega^{2k(2n+1)(2s+2t+1)}$, 
then 
\[
\dim \mathfrak{B}\left(\mathscr{G}_{k,p}^{st}\right)
=\left\{\begin{array}{ll}
4, & q=-1, (\text{Cartan type $A_1\times A_1$}),\\
27, & q^3=1\neq q, (\text{Cartan type $A_2$}),\\
\infty, & \text{otherwise}.
\end{array}\right.
\]
\end{lemma}
\begin{proof}
$\mathfrak{B}\left(\mathscr{G}_{k,p}^{st}\right)$
is of type $V_{abe}$ with 
$a=b=e=(-1)^p\lambda^{t+\frac32}\tilde{\mu}^{2s+2t+1}\omega^{2k(2n+1)(2s+2t+1)}$.
\end{proof}

\begin{lemma}\label{NA_H} 
$\mathfrak{B}\left(\mathscr{H}_{jk,p}^{st}\right)$
is of type $V_{abe}$ with 
$ae=\lambda\bar{\mu}^{2s+2t+1}\omega^{2k(2n+1)(4s+4t+2)+jN(4t+2)}$
and $b=(-1)^p\lambda^{t+\frac12}\bar{\mu}^{s+t+\frac12}
     \omega^{2k(2n+1)(2s+2t+1)-jN(2t+1)}$.
\end{lemma}

\begin{corollary}\label{Vabe}
Suppose $ae\neq b^2=(ae)^{-1}$ and $b\in\Bbb{G}_{m}$ for $m\geq 3$, then 
\[
\dim \mathfrak{B}(V_{abe})=\infty.
\]
\end{corollary}
\begin{proof}
If  $b\in\Bbb{G}_{4}$, then $b^2=(ae)^{-1}=-1=ae$.  
The case for $m\geq 5$ was proved in \cite[Corollary 4.15]{Shi2020even}. 
We only need to deal with the case $m=3$. Let $\mu=\lambda=N=1$, then
$ae=\omega^{2j(2t+1)}$, $b=(-1)^p \omega^{-j(2t+1)}$. And the parameter $b$ 
covers all $3$-th primitive 
roots of unity. For example, $b=\omega^{-12}\in\Bbb{G}_{3}$ in case $(n,j,t,p)=(4,4,1,0)$
and $b=\omega^{12}\in\Bbb{G}_{3}$ in case $(n,j,t,p)=(4,8,1,0)$. 
In this situation, $\dim \mathfrak{B}\left(\mathscr{H}_{jk,p}^{st}\right)=\infty$ since 
$A_{1n}^{++}$ is isomorphic to a $2$-cocycle deformation of $\widehat{D}_{4n}$
\cite{MR1800713} and Nichols algebras  associated with two dimensional Yetter-Drinfeld modules, over  the dihedral group $D_{4n}$ of order $4n$,  are either $4$-dimension or infinite dimension 
According to  \cite{andruskiewitsch2007pointed}. 
\end{proof}

\begin{corollary}
Let $N=n=1$, $\mu=1$, then $\dim \mathfrak{B}\left(\mathscr{H}_{jk,p}^{st}\right)=\infty$. 
\end{corollary}
\begin{proof}
In this situation, $ae=\lambda\omega^4$, $b=(-1)^p\lambda^{\frac12}\omega^{-2}$. 
So $ae\neq b^2=(ae)^{-1}$, $b\in\Bbb{G}_{6}$ in case $\lambda=1$ and 
$b\in\Bbb{G}_{12}$ in case $\lambda=-1$.
\end{proof}

\begin{remark}\label{special_remark}
From  observation, we have $\frac{ae}{b^2}=\omega^{4jN(2t+1)}$. 
\begin{enumerate}
\item When $ae=1$, then $b=\omega^{-4N}$ under the case 
$j=2$, $t=0$ and the suitable choice of $p$. 
So $\dim \mathfrak{B}\left(\mathscr{H}_{jk,p}^{st}\right)=(2n+1)^2$
for suitable choice of $(n,N,s,t,j,k,p)$.
Denote $\alpha=(n,N,s,t,j,k,p)$. For example, 
when $\mu=1$, $\lambda=-1$, $ae=1$, then
\begin{align*}
\dim \mathfrak{B}\left(\mathscr{H}_{jk,p}^{st}\right)=2^2, \quad
& \alpha=(4, 6, 1, 1, 6, 3, 0), (4, 6, 2, 1, 6, 3, 1),\cdots;\\
\dim \mathfrak{B}\left(\mathscr{H}_{jk,p}^{st}\right)=3^2, \quad
& \alpha=( 1 , 6 , 2 , 0 , 2 , 5 , 1 ), ( 1 , 6 , 3 , 0 , 2 , 1 , 0 ),\cdots;\\
\dim \mathfrak{B}\left(\mathscr{H}_{jk,p}^{st}\right)=5^2, \quad
& \alpha=( 2 , 10 , 1 , 0 , 4 , 7 , 1 ), ( 2 , 10 , 3 , 0 , 2 , 9 , 1 ),\cdots;\\
\dim \mathfrak{B}\left(\mathscr{H}_{jk,p}^{st}\right)=6^2, \quad
& \alpha=( 1 , 6 , 2 , 0 , 2 , 5 , 0 ), ( 1 , 6 , 5 , 0 , 2 , 5 , 1 ),\cdots;\\
\dim \mathfrak{B}\left(\mathscr{H}_{jk,p}^{st}\right)=7^2, \quad
& \alpha=( 3 , 14 , 1 , 0 , 2 , 11 , 1 ), ( 3 , 14 , 1 , 0 , 4 , 1 , 1 ),\cdots;\\
\dim \mathfrak{B}\left(\mathscr{H}_{jk,p}^{st}\right)=9^2, \quad
& \alpha=( 4 , 18 , 2 , 3 , 2 , 11 , 1 ), ( 4 , 18 , 2 , 3 , 8 , 17 , 1 ),\cdots;\\
\dim \mathfrak{B}\left(\mathscr{H}_{jk,p}^{st}\right)=10^2, \quad
& \alpha=( 2 , 10 , 3 , 0 , 4 , 3 , 0 ), ( 2 , 10 , 7 , 1 , 2 , 7 , 1 ),\cdots;\\
\dim \mathfrak{B}\left(\mathscr{H}_{jk,p}^{st}\right)=11^2, \quad
& \alpha=( 5 , 22 , 3 , 0 , 4 , 1 , 1 ), ( 5 , 22 , 3 , 1 , 8 , 1 , 1 ),\cdots;\\
\dim \mathfrak{B}\left(\mathscr{H}_{jk,p}^{st}\right)=13^2, \quad
& \alpha=( 6 , 26 , 26 , 5 , 6 , 7 , 1 ), ( 6 , 26 , 26 , 5 , 10 , 3 , 1 ),\cdots.
\end{align*} 
\item 
When $b=-1$, then $ae=\omega^{8N}$ in case $j=2$ and $t=0$. 
So $\dim \mathfrak{B}\left(\mathscr{H}_{jk,p}^{st}\right)=4(2n+1)$
for suitable choice of $(n,N,s,t,j,k,p)$.
Denote $\alpha=(n,N,s,t,j,k,p)$. 
For example, when $\mu=1$, $\lambda=-1$, $b=-1$, then
\begin{align*}
\dim \mathfrak{B}\left(\mathscr{H}_{jk,p}^{st}\right)=12, \quad
& \alpha=( 1 , 6 , 2 , 0 , 2 , 1 , 0 ), ( 1 , 18 , 1 , 0 , 2 , 11 , 1 ),\cdots;\\
\dim \mathfrak{B}\left(\mathscr{H}_{jk,p}^{st}\right)=20, \quad
& \alpha=( 2 , 10 , 1 , 0 , 2 , 9 , 0 ), ( 2 , 20 , 7 , 1 , 4 , 2 , 1 ),\cdots;\\
\dim \mathfrak{B}\left(\mathscr{H}_{jk,p}^{st}\right)=28, \quad
& \alpha=( 3 , 14 , 3 , 1 , 6 , 9 , 1 ), ( 3 , 28 , 25 , 1 , 4 , 6 , 1 ),\cdots;\\
\dim \mathfrak{B}\left(\mathscr{H}_{jk,p}^{st}\right)=36, \quad
& \alpha=( 4 , 18 , 4 , 2 , 8 , 1 , 0 ), ( 4 , 18 , 12 , 2 , 2 , 5 , 1 ),\cdots;\\
\dim \mathfrak{B}\left(\mathscr{H}_{jk,p}^{st}\right)=44, \quad
& \alpha=( 5 , 22 , 1 , 2 , 6 , 9 , 1 ), ( 5 , 22 , 2 , 0 , 2 , 7 , 1 ),\cdots;\\
\dim \mathfrak{B}\left(\mathscr{H}_{jk,p}^{st}\right)=52, \quad
& \alpha=( 6 , 26 , 1 , 2 , 2 , 7 , 1 ), ( 6 , 26 , 3 , 2 , 6 , 11 , 1 ),\cdots;\\
\dim \mathfrak{B}\left(\mathscr{H}_{jk,p}^{st}\right)=60, \quad
& \alpha=( 7 , 30 , 5 , 3 , 2 , 7 , 0 ), ( 7 , 30 , 6 , 0 , 4 , 13 , 1 ),\cdots.
\end{align*}
\end{enumerate}
\end{remark}

\begin{lemma}\label{NA_I}
$\mathfrak{B}\left(\mathscr{I}_{jk,p}^{st}\right)$
is of type $V_{abe}$ with 
$ae=\lambda\bar{\mu}^{2s+2t+1}\omega^{4k(2n+1)(2s+2t+1)-2jN(2t+1)}$
and $b=(-1)^p\lambda^{t-\frac12}\bar{\mu}^{s+t+\frac12}
     \omega^{2k(2n+1)(2s+2t+1)+jN(1+2t)}$.
\end{lemma}

\begin{corollary}
Let $N=n=1$, $\mu=1$, then $\mathfrak{B}\left(\mathscr{I}_{jk,p}^{st}\right)=\infty$. 
\end{corollary}
\begin{proof}
In this situation, $ae=\lambda \omega^{-4}$, $b=(-1)^p\lambda^{-\frac12}\omega^{2}$. 
So $b^2=(ae)^{-1}\neq ae$,  $b\in\Bbb{G}_{6}$ in case $\lambda=1$ and 
$b\in\Bbb{G}_{12}$ in case $\lambda=-1$. It's proved by the Corollary \ref{Vabe}. 
\end{proof}

\begin{remark}
Similar to the remark \ref{special_remark}, 
$\mathfrak{B}\left(\mathscr{I}_{jk,p}^{st}\right)
=\left\{\begin{array}{ll}
(2n+1)^2, &ae=1,\\
4(2n+1), &b=-1,
\end{array}\right.
$ for suitable choice of parameters. 
Denote $\alpha=(n,N,s,t,j,k,p)$. 
\begin{enumerate}
\item When $\mu=1$, $\lambda=-1$, $ae=1$, then
\begin{align*}
\dim \mathfrak{B}\left(\mathscr{I}_{jk,p}^{st}\right)=2^2, \quad
& \alpha=( 4 , 6 , 6 , 3 , 6 , 5 , 1 ), ( 4 , 10 , 1 , 1 , 6 , 5 , 1 ),\cdots;\\
\dim \mathfrak{B}\left(\mathscr{I}_{jk,p}^{st}\right)=3^2, \quad
& \alpha=( 1 , 6 , 3 , 0 , 2 , 5 , 1 ), ( 1 , 18 , 18 , 0 , 2 , 15 , 0 ),\cdots;\\
\dim \mathfrak{B}\left(\mathscr{I}_{jk,p}^{st}\right)=5^2, \quad
& \alpha=( 2 , 10 , 8 , 0 , 2 , 1 , 1 ), ( 2 , 10 , 9 , 1 , 2 , 1 , 0 ),\cdots;\\
\dim \mathfrak{B}\left(\mathscr{I}_{jk,p}^{st}\right)=6^2, \quad
& \alpha=( 1 , 12 , 3 , 0 , 2 , 10 , 0 ), ( 1 , 24 , 23 , 0 , 2 , 4 , 0 ),\cdots;\\
\dim \mathfrak{B}\left(\mathscr{I}_{jk,p}^{st}\right)=7^2, \quad
& \alpha=( 3 , 14 , 6 , 1 , 2 , 13 , 0 ), ( 3 , 14 , 8 , 1 , 4 , 1 , 1 ),\cdots;\\
\dim \mathfrak{B}\left(\mathscr{I}_{jk,p}^{st}\right)=9^2, \quad
& \alpha=( 4 , 18 , 4 , 2 , 2 , 7 , 0 ), ( 4 , 18 , 5 , 0 , 8 , 13 , 1 ),\cdots;\\
\dim \mathfrak{B}\left(\mathscr{I}_{jk,p}^{st}\right)=10^2, \quad
& \alpha=( 2 , 20 , 11 , 0 , 2 , 18 , 1 ), ( 2 , 20 , 12 , 1 , 4 , 2 , 1 ),\cdots;\\
\dim \mathfrak{B}\left(\mathscr{I}_{jk,p}^{st}\right)=11^2, \quad
& \alpha=( 5 , 22 , 1 , 1 , 6 , 19 , 0 ), ( 5 , 22 , 1 , 2 , 6 , 9 , 1 ),\cdots;\\
\dim \mathfrak{B}\left(\mathscr{I}_{jk,p}^{st}\right)=13^2, \quad
& \alpha=( 6 , 26 , 1 , 4 , 4 , 21 , 1 ), ( 6 , 26 , 2 , 1 , 4 , 11 , 1 ),\cdots.
\end{align*}
\item 
When $\mu=1$, $\lambda=-1$, $b=-1$, then
\begin{align*}
\dim \mathfrak{B}\left(\mathscr{I}_{jk,p}^{st}\right)=12, \quad
& \alpha=( 1 , 6 , 2 , 0 , 2 , 5 , 1 ), ( 1 , 18 , 7 , 0 , 2 , 5 , 1 ),\cdots;\\
\dim \mathfrak{B}\left(\mathscr{I}_{jk,p}^{st}\right)=20, \quad
& \alpha=( 2 , 10 , 2 , 1 , 4 , 9 , 1 ), ( 2 , 10 , 4 , 0 , 4 , 9 , 1 ),\cdots;\\
\dim \mathfrak{B}\left(\mathscr{I}_{jk,p}^{st}\right)=28, \quad
& \alpha=( 3 , 14 , 4 , 1 , 4 , 11 , 1 ), ( 3 , 14 , 5 , 2 , 6 , 5 , 0 ),\cdots;\\
\dim \mathfrak{B}\left(\mathscr{I}_{jk,p}^{st}\right)=36, \quad
& \alpha=( 4 , 18 , 5 , 3 , 2 , 5 , 1 ), ( 4 , 18 , 5 , 3 , 4 , 1 , 0 ),\cdots;\\
\dim \mathfrak{B}\left(\mathscr{I}_{jk,p}^{st}\right)=44, \quad
& \alpha=( 5 , 22 , 2 , 0 , 6 , 1 , 1 ), ( 5 , 22 , 2 , 0 , 8 , 5 , 0 ),\cdots;\\
\dim \mathfrak{B}\left(\mathscr{I}_{jk,p}^{st}\right)=52, \quad
& \alpha=( 6 , 26 , 1 , 4 , 8 , 23 , 1 ), ( 6 , 26 , 1 , 4 , 10 , 19 , 0 ),\cdots;\\
\dim \mathfrak{B}\left(\mathscr{I}_{jk,p}^{st}\right)=60, \quad
& \alpha=( 7 , 30 , 30 , 3 , 2 , 13 , 1 ), ( 7 , 30 , 30 , 3 , 14 , 1 , 1 ),\cdots.
\end{align*}
\end{enumerate}
\end{remark}

\subsection{The Nichols algebras $\mathfrak{B}\left(\mathscr{L}_{k,pq}^s\right)$}\label{section_NAL}
The braiding of $\mathfrak{B}\left(\mathscr{L}_{k,pq}^s\right)$ can be described as follows. 
\begin{enumerate}
\item When $a+b\equiv 1\mod 2$, 
let $b+(2a-1)=d(2n+1)+r$, $d\in\Bbb{N}$, $r\in\overline{0,2n}$. Then
\begin{align*}
c(m_a\otimes m_b)=
\left\{\begin{array}{rl}
\omega^{4k(2n+1)(s-a)} m_{b+2a-1}\otimes m_a, & d=0,\\
\omega^{4k(2n+1)(s-a)} m_{2n+1}\otimes m_a, & d=1, r=0\\
(-1)^p\omega^{2k(2n+1)[2(s-a)+2(r-1)+1]} m_{2n+1-(r-1)}\otimes m_a, & d=1, r\neq 0,\\
(-1)^p\omega^{2k(2n+1)[2(s-a)+4n+1]} m_{1}\otimes m_a, & d=2, r=0,\\
\omega^{4k(2n+1)[(s-a)+2n+1]} m_{r}\otimes m_a, & d=2, r\neq 0.
\end{array}\right.
\end{align*}
\item When $a+b\equiv 0\mod 2$, let $2a-1=b+1+d(2n+1)+r$, 
$d\in \Bbb{N}$, $r\in\overline{0,2n}$.
Then 
\begin{align*}
 c(m_a\otimes m_b)=
\left\{\begin{array}{rl}
\omega^{4k(2n+1)(s-a+2a-1)}m_{b-(2a-1)}\otimes m_a, & 2a-1<b,\\
(-1)^p\omega^{2k(2n+1)[2(s-a)+2b-1]}m_{1}\otimes m_a,&2a-1=b,\\
(-1)^p\omega^{2k(2n+1)[2(s-a)+2b-1]}m_{r+1}\otimes m_a, 
                 &d=0,\\
\omega^{4k(2n+1)[s-a+b+r]}m_{2n+1-r}\otimes m_a, 
                &d=1.\\
\end{array}\right.
\end{align*}
\end{enumerate}

\begin{lemma}\label{racklemma}
Let $(V,c)$ be a braided vector space such that 
$c(x\otimes y)\in \k f_x(y)\otimes x$, where the map 
$f_x: V\to V$ is bijective for any $x\in V$ under a fixed basis. Then $(V,\vartriangleright)$ is a rack 
with $x\vartriangleright y=f_x(y)$ under the fixed basis. 
\end{lemma}
\begin{proof}
For any $x,y,z\in V$, 
\begin{align*}
c_1c_2c_1(x\otimes y\otimes z)
&\in \k [(x\vartriangleright y)\vartriangleright (x \vartriangleright z)]
\otimes (x\vartriangleright y)\otimes x, \\
c_2c_1c_2(x\otimes y\otimes z)
&\in\k [x\vartriangleright (y\vartriangleright z)]
\otimes (x\vartriangleright y)\otimes x. 
\end{align*}
So $x\vartriangleright (y\vartriangleright z)
=(x\vartriangleright y)\vartriangleright (x \vartriangleright z)$.
\end{proof}

\begin{corollary}
The Nichols algebra $\mathfrak{B}\left(\mathscr{L}_{k,pq}^s\right)$
is of rack type. 
\end{corollary}

\begin{lemma}
When $n=1$, $p=1$, $k=0$, then $\dim \mathfrak{B}\left(\mathscr{L}_{k,pq}^s\right)=12$. 
It is generated by $m_1$, $m_2$, $m_3$ and with relations
\begin{align}\label{Ndim12}
m_3m_2=m_1m_3-m_2m_1,\quad m_2m_3=-m_1m_2+m_3m_1, 
\quad m_i^2=0 \quad\forall i\in\overline{1,3}.
\end{align}
\end{lemma}
\begin{proof}
When $n=1$, then the braiding of $\mathfrak{B}\left(\mathscr{L}_{k,pq}^s\right)$ is given by 
\begin{align*}
c(m_1\otimes m_1)
&=(-1)^p\omega^{2k(2n+1)(2s-1)}m_1\otimes m_1,\\
c(m_1\otimes m_2)
&=\omega^{4k(2n+1)(s-1)}m_3\otimes m_1,\\
c(m_1\otimes m_3)
&=\omega^{4k(2n+1)s}m_2\otimes m_1,\\
c(m_2\otimes m_1)
&=(-1)^p\omega^{2k(2n+1)[1+2(s-2)]}m_3\otimes m_2,\\
c(m_2\otimes m_2)
&=(-1)^p\omega^{2k(2n+1)(2s-1)}m_2\otimes m_2,\\
c(m_2\otimes m_3)
&=(-1)^p\omega^{2k(2n+1)(1+2s)}m_1\otimes m_2,\\
c(m_3\otimes m_1)
&=\omega^{4k(2n+1)(s-1)}m_2\otimes m_3,\\
c(m_3\otimes m_2)
&=\omega^{4k(2n+1)s}m_1\otimes m_3,\\
c(m_3\otimes m_3)
&=(-1)^p\omega^{2k(2n+1)(2s-1)}m_3\otimes m_3.
\end{align*}
When $p=1$, $k=0$, it's easy to see that the relations \eqref{Ndim12} hold. 
The Nichols algebra is isomorphic to the $12$-dimensional Nichols algebra over 
the dehidral group $D_{6}$ of order $6$, see \cite[Proposition 3.3.9]{andruskiewitsch1998braided}. 
\end{proof}

\begin{corollary}
When $N=n=1$, then $\dim \mathfrak{B}\left(\mathscr{L}_{k,pq}^s\right)
=\left\{\begin{array}{ll}12, &\text{if}\,\,p=1,\\
\infty, &\text{if}\,\,p=0.\end{array}\right.$
\end{corollary}

\begin{lemma}
Suppose  $n=2$, $k=0$, $p=1$, then the Nichols algebra 
$\mathfrak{B}\left(\mathscr{L}_{k,pq}^s\right)$ has the following relations.
\begin{align*}
m_i^2=0,\quad \forall i\in\overline{1,5},\\
m_1m_2-m_3m_1-m_5m_3+m_4m_5-m_2m_4=0,\\
m_1m_3-m_2m_1+m_4m_2-m_5m_4+m_3m_5=0,\\
m_1m_4-m_5m_1+m_2m_5+m_3m_2+m_4m_3=0,\\
m_1m_5-m_4m_1-m_3m_4-m_2m_3-m_5m_2=0.
\end{align*}
\end{lemma}
\begin{proof}
When $n=2$, then the braiding of $\mathfrak{B}\left(\mathscr{L}_{k,pq}^s\right)$ is given by 
\begin{align*}
c( m_1 \otimes m_1 )&= A B^{2 s - 1} m_1 \otimes m_1,
&c( m_1 \otimes m_2 )&= B^{2 s - 2} m_3 \otimes m_1, \\
c( m_1 \otimes m_3 )&= B^{2 s} m_2 \otimes m_1, 
&c( m_1 \otimes m_4 )&= B^{2 s - 2} m_5 \otimes m_1, \\
c( m_1 \otimes m_5 )&= B^{2 s} m_4 \otimes m_1, 
&c( m_2 \otimes m_1 )&= B^{2 s - 4} m_4 \otimes m_2, \\
c( m_2 \otimes m_2 )&= A B^{2 s - 1} m_2 \otimes m_2, 
&c( m_2 \otimes m_3 )&= A B^{2 s - 3} m_5 \otimes m_2, \\
c( m_2 \otimes m_4 )&= B^{2 s + 2} m_1 \otimes m_2, 
&c( m_2 \otimes m_5 )&= A B^{2 s + 1} m_3 \otimes m_2, \\
c( m_3 \otimes m_1 )&= A B^{2 s - 5} m_5 \otimes m_3, 
&c( m_3 \otimes m_2 )&= A B^{2 s - 3} m_4 \otimes m_3, \\
c( m_3 \otimes m_3 )&= A B^{2 s - 1} m_3 \otimes m_3, 
&c( m_3 \otimes m_4 )&= A B^{2 s + 1} m_2 \otimes m_3, \\
c( m_3 \otimes m_5 )&= A B^{2 s + 3} m_1 \otimes m_3, 
&c( m_4 \otimes m_1 )&= A B^{2 s - 3} m_3 \otimes m_4, \\
c( m_4 \otimes m_2 )&=  B^{2 s - 4} m_5 \otimes m_4, 
&c( m_4 \otimes m_3 )&= A B^{2 s + 1} m_1 \otimes m_4, \\
c( m_4 \otimes m_4 )&= A B^{2 s - 1} m_4 \otimes m_4, 
&c( m_4 \otimes m_5 )&=  B^{2 s + 2} m_2 \otimes m_4, \\
c( m_5 \otimes m_1 )&=  B^{2 s - 2} m_2 \otimes m_5, 
&c( m_5 \otimes m_2 )&=  B^{2 s} m_1 \otimes m_5, \\
c( m_5 \otimes m_3 )&=  B^{2 s - 2} m_4 \otimes m_5, 
&c( m_5 \otimes m_4 )&=  B^{2 s} m_3 \otimes m_5, \\
c( m_5 \otimes m_5 )&= A B^{2 s - 1} m_5 \otimes m_5,
\end{align*}
where $A=(-1)^p$, $B=\omega^{2k(2n+1)}$. So $A=-1$, $B=1$ in case 
$k=0$, $p=1$.
The relations can be obtained by direct computations.
\end{proof}

\subsection{The Nichols algebra $\mathfrak{B}\left(\mathscr{N}_{k,pq}^s\right)$} \label{section_NAN}
 
 Define $R^{\gamma }: \mathscr{N}_{k,pq}^s\otimes \mathscr{N}_{k,pq}^s
\longrightarrow \mathscr{N}_{k,pq}^s\otimes \mathscr{N}_{k,pq}^s$ such that 
$ R^{\gamma }(w_\alpha\otimes w_\beta)=$
\begin{align*}
\left\{\begin{array}{rl}
w_{\beta+\gamma }\otimes w_{2n-\alpha+2},
&\beta+\gamma \leq 2n+1,\\
(-1)^p\lambda\bar{\mu}^{\frac12}\omega^{2k(2n+1)}
w_{2n+1}\otimes w_{2n-\alpha+2},
&\beta+\gamma =2n+2,\\
(-1)^p\lambda\left[\bar{\mu}^{\frac12}\omega^{2k(2n+1)}\right]^{2(\gamma +\beta)-4n-3}
w_{4n+3-\gamma -\beta}\otimes w_{2n-\alpha+2},
&\beta+\gamma \in\overline{2n+3,4n+2},\\
\lambda\left[\bar{\mu}^{\frac12}\omega^{2k(2n+1)}\right]^{4n+2}
w_{1+\beta+\gamma-(4n+3)}\otimes w_{2n-\alpha+2},
&\beta+\gamma\geq 4n+3,
\end{array}\right.
\end{align*}
and   $L^{\gamma }: \mathscr{N}_{k,pq}^s\otimes \mathscr{N}_{k,pq}^s
\longrightarrow \mathscr{N}_{k,pq}^s\otimes \mathscr{N}_{k,pq}^s$ such that 
\begin{align*}
L^{\gamma }(w_\alpha\otimes w_\beta)=
\left\{\begin{array}{rl}
\left[\bar{\mu}^{\frac12}\omega^{2k(2n+1)}\right]^{2\gamma}
w_{\beta-\gamma}\otimes w_{2n-\alpha+2}, 
&\gamma<\beta,\vspace{1mm}\\
(-1)^p\left[\bar{\mu}^{\frac12}\omega^{2k(2n+1)}\right]^{2\beta-1}
w_{\gamma-\beta+1}\otimes w_{2n-\alpha+2}, 
&\gamma\in\overline{\beta,\beta+2n},\\
\lambda\left[\bar{\mu}^{\frac12}\omega^{2k(2n+1)}\right]^{2\beta}
w_{2n+1}\otimes w_{2n-\alpha+2}, 
&\gamma=\beta+2n+1,\\
\lambda\left[\bar{\mu}^{\frac12}\omega^{2k(2n+1)}\right]^{2(\gamma-2n-1)}
w_{4n+2+\beta-\gamma}\otimes w_{2n-\alpha+2}, 
&\gamma>\beta+2n+1.
\end{array}\right.
\end{align*}
Then the braiding of $\mathfrak{B}\left(\mathscr{N}_{k,pq}^s\right)$ can be described as follows. 
\begin{enumerate}
\item When $\alpha=n+1$, then 
\[
c(w_\alpha\otimes w_\beta)
=(-1)^q\left[\bar{\mu}^{\frac12}\omega^{2k(2n+1)}\right]^{2(\alpha+s-1)}
w_\beta\otimes w_{2n-\alpha+2};
\]
\item When $\alpha<n+1$, 
          $\alpha+\beta \equiv 0\mod 2$,
         then 
\[
c(w_\alpha\otimes w_\beta)
=(-1)^q\left[\bar{\mu}^{\frac12}\omega^{2k(2n+1)}\right]^{2(2\alpha+s-n-2)}
L^{2(n-\alpha+1)}(w_\alpha\otimes w_\beta);
\]
\item When $\alpha>n+1$, 
          $\alpha+\beta\equiv 1\mod 2$,
         then 
\[
c(w_\alpha\otimes w_\beta)
=(-1)^q\left[\bar{\mu}^{\frac12}\omega^{2k(2n+1)}\right]^{2(s+n)}
L^{2(\alpha-1-n)}(w_\alpha\otimes w_\beta);
\]

\item When $\alpha<n+1$, 
          $\alpha+\beta \equiv 1\mod 2$,
         then 
\[
c(w_\alpha\otimes w_\beta)
=(-1)^q\left[\bar{\mu}^{\frac12}\omega^{2k(2n+1)}\right]^{2(2\alpha+s-n-2)}
R^{2(n-\alpha+1)}(w_\alpha\otimes w_\beta);
\]
\item When $\alpha>n+1$, 
          $\alpha+\beta\equiv 0\mod 2$,
         then 
\[
c(w_\alpha\otimes w_\beta)
=(-1)^q\left[\bar{\mu}^{\frac12}\omega^{2k(2n+1)}\right]^{2(s+n)}
R^{2(\alpha-1-n)}(w_\alpha\otimes w_\beta).
\]
\end{enumerate}

\begin{lemma}
When $n=1=\mu$, $k=0$, $q=1$, then $\dim \mathfrak{B}\left(\mathscr{N}_{k,pq}^s\right)=12$. 
It is generated by $w_1$, $w_2$, $w_3$ and with relations:
\begin{align}\label{NA12Different1}
w_1^2+(-1)^pw_2w_3+(-1)^p w_3w_2=0,\quad  w_2^2=0 , \\
w_1w_2+\lambda (-1)^pw_3^2+w_2w_1=0,\quad w_1w_3=0,\quad w_3w_1=0. \label{NA12Different2}
\end{align}
\end{lemma}
\begin{proof}
When $n=1$, the braiding of $\mathfrak{B}\left(\mathscr{N}_{k,pq}^s\right)$ is given by 
\begin{align*}
c( w_1 \otimes w_1 )
&= (-1)^{p+q}\left[\bar{\mu}^{\frac12}\omega^{2k(2n+1)}\right]^{2 s - 1}  w_2 \otimes w_3,\\
c( w_1 \otimes w_2 )
&=\lambda(-1)^{p+q}\left[\bar{\mu}^{\frac12}\omega^{2k(2n+1)}\right]^{2 s - 1} w_3 \otimes w_3, \\
c( w_1 \otimes w_3 )
&=(-1)^q \left[\bar{\mu}^{\frac12}\omega^{2k(2n+1)}\right]^{2 s + 2}  w_1 \otimes w_3, \\
c( w_2 \otimes w_1 )
&=(-1)^q\left[\bar{\mu}^{\frac12}\omega^{2k(2n+1)}\right]^{2 s + 2}  w_1 \otimes w_2, \\
c( w_2 \otimes w_2 )
&=(-1)^q\left[\bar{\mu}^{\frac12}\omega^{2k(2n+1)}\right]^{2 s + 2}  w_2 \otimes w_2, \\
c( w_2 \otimes w_3 )
&=(-1)^q\left[\bar{\mu}^{\frac12}\omega^{2k(2n+1)}\right]^{2 s + 2}  w_3 \otimes w_2, \\
c( w_3 \otimes w_1 )
&=(-1)^q\left[\bar{\mu}^{\frac12}\omega^{2k(2n+1)}\right]^{2 s + 2} w_3 \otimes w_1, \\
c( w_3 \otimes w_2 )
&=(-1)^{p+q}\left[\bar{\mu}^{\frac12}\omega^{2k(2n+1)}\right]^{2 s + 5} w_1 \otimes w_1, \\
c( w_3 \otimes w_3 )
&=\lambda(-1)^{p+q}\left[\bar{\mu}^{\frac12}\omega^{2k(2n+1)}\right]^{2 s + 5}w_2 \otimes w_1.
\end{align*}
It's easy to see that the relations \eqref{NA12Different1} and \eqref{NA12Different2} hold in case 
$\mu=1$, $k=0$ and $q=1$. According to direct computation with the relations, we have 
\begin{align*}
w_1^3=(-1)^{p+1}w_1w_2w_3=(-1)^{p+1}w_3w_2w_1=\lambda w_3^3,\\
w_1^2w_2=(-1)^{p+1}w_2w_3w_2=w_2w_1^2=-w_1w_2w_1,\\
w_2w_3^2=-w_3w_2w_3=w_3^2w_2=\lambda (-1)^{p+1}w_2w_1w_2, 
\end{align*}
and the other monomial in degree $3$ vanish. This in turn  implies 
\begin{align*}
w_1^4&=w_1^2w_2w_3=w_3w_2w_1^2
              =w_3^4=w_3^2w_2w_1=w_1w_2w_3^2
              =w_2w_1^2w_2\\
           &=w_2w_1w_2w_1=w_1w_2w_1w_2
              =w_3w_2w_3w_2=w_2w_3w_2w_3=w_2w_3^2w_2=0,\\
w_2w_1^3&=(-1)^{p+1}w_2w_1w_2w_3=(-1)^{p+1}w_2w_3w_2w_1=\lambda w_2w_3^3\\
&=-\lambda w_3w_2w_3^2=\lambda w_3^2w_2w_3=-\lambda w_3^3w_2
=(-1)^p w_3w_2w_1w_2 = w_1^2w_2w_1 \\
&=-w_1w_2w_1^2=-w_1^3w_2= (-1)^p w_1w_2w_3w_2,        
\end{align*}
and the other monomials in degree $4$ vanish. Now it's easy to see that the Nichols algebra 
has a basis given by 
$
\left\{1, w_1,w_2,w_3, w_1^3, w_1^2w_2, w_2w_3^2,  w_2w_1^3\right\}.
$
\end{proof}

\begin{corollary}
When $N=n=1=\mu$,  then $\dim \mathfrak{B}\left(\mathscr{N}_{k,pq}^s\right)=
\left\{\begin{array}{ll}12, &q=1,\\ \infty, &q=0.\end{array}\right.$
\end{corollary}

\begin{lemma}
When $n=2$, $q=1=\mu$, $k=0$, then the Nichols algebra 
$\mathfrak{B}\left(\mathscr{N}_{k,pq}^s\right)$ has the following relations. 
\begin{align*}
w_1w_5=0,\quad w_2w_4=0,\quad w_3^2=0,\quad  w_4w_2=0,\quad w_5w_1=0,\\
(-1)^p w_1^2+w_3w_2+w_2w_3+w_4w_5+w_5w_4=0,\\
w_1w_2+\lambda (-1)^p w_5^2+w_2w_1+w_3w_4+w_4w_3=0,\\
w_1w_3+(-1)^p w_2w_5+\lambda w_4^2+(-1)^p w_5w_2+w_3w_1=0,\\
w_1w_4+\lambda (-1)^p w_3w_5+ \lambda (-1)^p w_5w_3+ w_4w_1+w_2^2=0.
\end{align*}
\end{lemma}
\begin{proof}
When $n=2$, the braiding of $\mathfrak{B}\left(\mathscr{N}_{k,pq}^s\right)$ is given by 
\begin{align*}
c( w_1 \otimes w_1 )&= (-1)^{p+q} \left[\bar{\mu}^{\frac12}\omega^{2k(2n+1)}\right]^{2 s - 3}  w_4 \otimes w_5, \\
c( w_1 \otimes w_2 )&= \lambda(-1)^{p+q}  \left[\bar{\mu}^{\frac12}\omega^{2k(2n+1)}\right]^{2 s - 3}  w_5 \otimes w_5, \\
c( w_1 \otimes w_3 )&= (-1)^{p+q}  \left[\bar{\mu}^{\frac12}\omega^{2k(2n+1)}\right]^{2 s + 1}  w_2 \otimes w_5, \\
c( w_1 \otimes w_4 )&= \lambda(-1)^{p+q}  \left[\bar{\mu}^{\frac12}\omega^{2k(2n+1)}\right]^{2 s + 1}  w_3 \otimes w_5, \\
c( w_1 \otimes w_5 )&= (-1)^q \left[\bar{\mu}^{\frac12}\omega^{2k(2n+1)}\right]^{2 s + 4}  w_1 \otimes w_5, \\
c( w_2 \otimes w_1 )&= (-1)^q \left[\bar{\mu}^{\frac12}\omega^{2k(2n+1)}\right]^{2 s}  w_3 \otimes w_4, \\
c( w_2 \otimes w_2 )&= (-1)^{p+q} \left[\bar{\mu}^{\frac12}\omega^{2k(2n+1)}\right]^{2 s + 3}  w_1 \otimes w_4, \\
c( w_2 \otimes w_3 )&= (-1)^q \left[\bar{\mu}^{\frac12}\omega^{2k(2n+1)}\right]^{2 s}  w_5 \otimes w_4, \\
c( w_2 \otimes w_4 )&= (-1)^q \left[\bar{\mu}^{\frac12}\omega^{2k(2n+1)}\right]^{2 s + 4}  w_2 \otimes w_4, \\
c( w_2 \otimes w_5 )&= \lambda(-1)^{p+q} \left[\bar{\mu}^{\frac12}\omega^{2k(2n+1)}\right]^{2 s + 3}  w_4 \otimes w_4, \\
c( w_3 \otimes w_1 )&= (-1)^q \left[\bar{\mu}^{\frac12}\omega^{2k(2n+1)}\right]^{2 s + 4}  w_1 \otimes w_3, \\
c( w_3 \otimes w_2 )&= (-1)^q \left[\bar{\mu}^{\frac12}\omega^{2k(2n+1)}\right]^{2 s + 4}  w_2 \otimes w_3, \\
c( w_3 \otimes w_3 )&= (-1)^q \left[\bar{\mu}^{\frac12}\omega^{2k(2n+1)}\right]^{2 s + 4}  w_3 \otimes w_3, \\
c( w_3 \otimes w_4 )&= (-1)^q \left[\bar{\mu}^{\frac12}\omega^{2k(2n+1)}\right]^{2 s + 4}  w_4 \otimes w_3, \\
c( w_3 \otimes w_5 )&= (-1)^q \left[\bar{\mu}^{\frac12}\omega^{2k(2n+1)}\right]^{2 s + 4}  w_5 \otimes w_3, \\
c( w_4 \otimes w_1 )&= (-1)^{p+q}  \left[\bar{\mu}^{\frac12}\omega^{2k(2n+1)}\right]^{2 s + 5}  w_2 \otimes w_2, \\
c( w_4 \otimes w_2 )&= (-1)^q \left[\bar{\mu}^{\frac12}\omega^{2k(2n+1)}\right]^{2 s + 4}  w_4 \otimes w_2, \\
c( w_4 \otimes w_3 )&= (-1)^q \left[\bar{\mu}^{\frac12}\omega^{2k(2n+1)}\right]^{2 s + 8}  w_1 \otimes w_2, \\
c( w_4 \otimes w_4 )&= \lambda(-1)^{p+q}  \left[\bar{\mu}^{\frac12}\omega^{2k(2n+1)}\right]^{2 s + 5}  w_5 \otimes w_2, \\
c( w_4 \otimes w_5 )&= (-1)^q \left[\bar{\mu}^{\frac12}\omega^{2k(2n+1)}\right]^{2 s + 8}  w_3 \otimes w_2, \\
c( w_5 \otimes w_1 )&= (-1)^q \left[\bar{\mu}^{\frac12}\omega^{2k(2n+1)}\right]^{2 s + 4}  w_5 \otimes w_1, \\
c( w_5 \otimes w_2 )&= (-1)^{p+q} \left[\bar{\mu}^{\frac12}\omega^{2k(2n+1)}\right]^{2 s + 7}  w_3 \otimes w_1, \\
c( w_5 \otimes w_3 )&= \lambda(-1)^{p+q}  \left[\bar{\mu}^{\frac12}\omega^{2k(2n+1)}\right]^{2 s + 7}  w_4 \otimes w_1, \\
c( w_5 \otimes w_4 )&= (-1)^{p+q} \left[\bar{\mu}^{\frac12}\omega^{2k(2n+1)}\right]^{2 s + 11}  w_1 \otimes w_1, \\
c( w_5 \otimes w_5 )&= \lambda(-1)^{p+q}  \left[\bar{\mu}^{\frac12}\omega^{2k(2n+1)}\right]^{2 s + 11}  w_2 \otimes w_1.
\end{align*}
The relations can be checked via direct computations. 
\end{proof}

\subsection{Finite dimensional Nichols algebras over $A_{13}^{+\lambda}$}
\begin{theorem}
Finite dimensional Nichols algebras over $A_{13}^{+\lambda}$, associated with simple 
Yetter-Drinfeld modules, can be classified as follows. 
\begin{enumerate}
\item $\dim\mathfrak{B}\left(\mathscr{B}_{k,p}^s\right)
        =\left\{\begin{array}{ll}2,&\text{if}\,\,\lambda=1, s=1, k=0, p=1,\\
         4, &\text{if}\,\, \lambda=-1, s=1, k=0, p\in\Bbb{Z}_2.\end{array}\right.$
\item $\dim \mathfrak{B}\left(\mathscr{F}_{k,p}^{st}\right)=4$ in case  
          $s=1$, $t=0$, $k=1$, $p\in\Bbb{Z}_2$.          
\item $\dim \mathfrak{B}\left(\mathscr{G}_{k,p}^{st}\right)=4$ in case   $\lambda=1=s$, $t=0$, 
         $(p,k)=(0,1)$ or $(1,0)$. 
\item  $\dim \mathfrak{B}\left(\mathscr{L}_{k,pq}^s\right)
         =12$ in case $s=p=1$, $k=0$, $q\in\Bbb{Z}_2$. 
\item $\dim \mathfrak{B}\left(\mathscr{N}_{k,pq}^s\right)=12$ in case    
        $s=q=1$, $k=0$, $p\in\Bbb{Z}_2$.     
\end{enumerate}
\end{theorem}
\begin{remark}
The proof is just a summation of the section. We will deal with the 
classification of finite dimensional Hopf algebras over $A_{13}^{+\lambda}$ in a sequel\cite{Shi}.
\end{remark}

\section{Appendix}
\subsection{One dimensional Yetter-Drinfeld modules over $A_{N\,2n+1}^{\mu\lambda}$}
\begin{enumerate}
\item $\mathscr{A}_{k,p}^s=\k w$, where 
$
w=v\boxtimes \left[x_{11}^{2s}+(-1)^p x_{12}^{2s}\right],
$
$s\in\overline{1,N}$, $V_{k}=\k v$, $p\in\Bbb{Z}_2$.
\item $\mathscr{B}_{k,p}^s=\k w$, where 
$w=v\boxtimes \left[x_{11}^{2s+1}\chi_{22}^{2n}+(-1)^p\sqrt{\lambda}x_{12}^{2s+1}\chi_{21}^{2n}\right]$, $V_k^\prime=\k v$, $s\in\overline{1,N}$, $p\in\Bbb{Z}_2$.
\end{enumerate}

\subsection{Two dimensional Yetter-Drinfeld modules over $A_{N\,2n+1}^{\mu\lambda}$}
\begin{enumerate}
\item $\mathscr{C}_{jk,p}^{st}=\k w_1\oplus \k w_2$, where 
\begin{align*}
w_1&=v_1\boxtimes x_{11}^{2s}\chi_{22}^{2t+2}
    +(-1)^p\omega^{-2k(2n+1)+jN} v_2\boxtimes x_{12}^{2s}\chi_{21}^{2t+2}\\
w_2&=v_2\boxtimes x_{11}^{2s+1}\chi_{22}^{2t+1}
           +(-1)^p\omega^{2k(2n+1)-jN}v_1\boxtimes x_{12}^{2s+1}\chi_{21}^{2t+1},
\end{align*}
$s\in\overline{1,N}$, $p\in\Bbb{Z}_2$, $t\in\overline{0,n-1}$, $V_{jk}=\k v_1\oplus\k v_2$.

\item $\mathscr{D}_{jk,p}^{st}=\k w_1\oplus \k w_2$, where  
\begin{align*}
w_1&=v_1\boxtimes x_{11}^{2s+1}\chi_{22}^{2t+1}
        +(-1)^p\omega^{jN-2k(2n+1)}v_2\boxtimes x_{12}^{2s+1}\chi_{21}^{2t+1},\\
w_2&=v_2\boxtimes x_{11}^{2s}\chi_{22}^{2t+2}
        +(-1)^p\omega^{2k(2n+1)-jN}v_1\boxtimes x_{12}^{2s}\chi_{21}^{2t+2},
\end{align*}
$s\in\overline{1,N}$, $p\in\Bbb{Z}_2$, $t\in\overline{0,n-1}$,
$V_{jk}=\k v_1\oplus\k v_2$.

\item $\mathscr{E}_{jk,p}^s=\k w_1\oplus \k w_2$, where 
\begin{align*}
w_1&=v_1\boxtimes x_{11}^{2s}+(-1)^p\omega^{jN-2k(2n+1)}v_2\boxtimes x_{12}^{2s},\\
w_2&=v_2\boxtimes x_{11}^{2s}+(-1)^p\omega^{2k(2n+1)-jN}v_1\boxtimes x_{12}^{2s},
\end{align*}
$s\in\overline{1,N}$, $p\in\Bbb{Z}_2$, $V_{jk}=\k v_1\oplus\k v_2$.

\item  $\mathscr{F}_{k,p}^{st}=\k w_1\oplus\k w_2$, where  
\begin{align*}
w_1&=v\boxtimes \left[x_{11}^{2s+1}\chi_{22}^{2t+1}+(-1)^px_{12}^{2s+1}\chi_{21}^{2t+1}\right],\\
w_2&=v\boxtimes \left[x_{11}^{2s}\chi_{22}^{2t+2}+(-1)^px_{12}^{2s}\chi_{21}^{2t+2}\right],
\end{align*}
$s\in\overline{1,N}$, $t\in\overline{0,n-1}$, $V_{k}=\k v$, $p\in\Bbb{Z}_2$.

\item $\mathscr{G}_{k,p}^{st}=\k w_1\oplus \k w_2$, where 
\begin{align*}
w_1&=v\boxtimes\left[x_{11}^{2s}\chi_{22}^{2t+1}
          +(-1)^p\sqrt{\lambda}x_{12}^{2s}\chi_{21}^{2t+1}\right],\\
w_2&=v\boxtimes\left[x_{11}^{2s+1}\chi_{22}^{2t}
        +\frac{(-1)^p}{\sqrt{\lambda}}x_{12}^{2s+1}\chi_{21}^{2t}\right],
\end{align*}
$s\in\overline{1,N}$, $t\in\overline{0,n-1}$, $p\in\Bbb{Z}_2$, $V_k^\prime=\k v$.  

\item $\mathscr{H}_{jk,p}^{st}=\k w_1\oplus \k w_2$, where  
\begin{align*}
w_1&=v_1^\prime\boxtimes x_{11}^{2s}\chi_{22}^{2t+1}
            +\frac{(-1)^p}{\sqrt{\lambda\bar{\mu}}}\omega^{jN-2k(2n+1)}
           v_2^\prime\boxtimes x_{12}^{2s}\chi_{21}^{2t+1},\\
w_2&=v_2^\prime\boxtimes x_{11}^{2s+1}\chi_{22}^{2t}
           +(-1)^p\sqrt{\lambda\bar{\mu}}\omega^{2k(2n+1)-jN}
           v_1^\prime\boxtimes x_{12}^{2s+1}\chi_{21}^{2t},
\end{align*}
$V_{jk}^\prime=\k v_1^\prime\oplus\k v_2^\prime$, 
$p\in\Bbb{Z}_2$, $s\in\overline{1,N}$, $t\in\overline{0,n-1}$.

\item $\mathscr{I}_{jk,p}^{st}=\k\,w_1\oplus \k\,w_2$, where 
\begin{align*}
w_1&=v_1^\prime\boxtimes x_{11}^{2s+1}\chi_{22}^{2t}+\frac{(-1)^p}{\sqrt{\lambda\bar{\mu}}}\omega^{jN-2k(2n+1)}v_2^\prime
           \boxtimes x_{12}^{2s+1}\chi_{21}^{2t},\\
w_2&=v_2^\prime\boxtimes x_{11}^{2s}\chi_{22}^{2t+1}
    +(-1)^p\sqrt{\lambda\bar{\mu}}\omega^{2k(2n+1)-jN}v_1^\prime
    \boxtimes x_{12}^{2s}\chi_{21}^{2t+1},
\end{align*}
$V_{jk}^\prime=\k v_1^\prime\oplus\k v_2^\prime$, 
$p\in\Bbb{Z}_2$, $s\in\overline{1,N}$, $t\in\overline{0,n-1}$.
\end{enumerate}

\subsection{$2n+1$ dimensional Yetter-Drinfeld modules over $A_{N\,2n+1}^{\mu\lambda}$}
\begin{enumerate}
\item Let $s\in\overline{1,N}$, $p,q\in\Bbb{Z}_2$, 
$V_{jk}=\k v_1\oplus\k v_2$, and denote 
\begin{align*}
a&=(-1)^p\omega^{2jN-2k(2n+1)},\quad b=(-1)^q\sqrt{\lambda}\omega^{-4kn(2n+1)}, \\
m_1&= \left(v_1+av_2\right)\boxtimes x_{11}^{2s-1}
+b\chi_{22}^{2n}\cdot \left[\left(v_1+av_2\right)\boxtimes x_{21}^{2s-1}\right],\\
m_i&= \left\{\begin{array}{rl}
x_{22}\cdot m_{i-1}=\chi_{22}^{i-1}\cdot m_1, 
& i\,\,\text{is even}, 1<i\leq 2n+1,\vspace{1mm}\\
x_{11}\cdot m_{i-1}=\chi_{11}^{i-1}\cdot m_1, 
& i\,\,\text{is odd}, 1<i\leq 2n+1.
\end{array}\right.
\end{align*}
Then $\mathscr{L}_{k,pq}^s=\bigoplus_{i=1}^{2n+1}\k m_i$ is a $2n+1$ dimensional Yetter-Drinfeld module over $A_{N\,2n+1}^{\mu\lambda}$ with  the module structure given by
\begin{align*}
x_{11}\cdot m_i
&=\left\{\begin{array}{rl}
           \chi_{11}^{i}\cdot w_1=m_{i+1}, & i\,\,\text{is even}, 1<i\leq 2n+1,\vspace{1mm}\\
           \chi_{22}^{i-2}x_{11}^2\cdot m_1=\omega^{4k(2n+1)}m_{i-1}, 
                                                    & i\,\,\text{is odd}, 1<i\leq 2n+1, \\
           (-1)^p\omega^{2k(2n+1)} m_1, & i=1,\vspace{1mm}\\
   \end{array}\right.\\
x_{22}\cdot m_i
&=\left\{\begin{array}{rl}
           \chi_{11}^{i-2}\cdot m_1=\omega^{4k(2n+1)}m_{i-1}, 
                                     & i\,\,\text{is even}, 1<i< 2n+1,\vspace{1mm}\\
           \chi_{22}^{i}\cdot m_1=m_{i+1}, & i\,\,\text{is odd}, 1\leq i< 2n+1, \vspace{1mm}\\
           \chi_{22}^{2n+1}\cdot m_1=(-1)^p \omega^{2k(2n+1)}m_{2n+1}, & i=2n+1,\vspace{1mm}\\
   \end{array}\right.\\
x_{pq}\cdot m_i
&=0, \quad pq=12\,\,\text{or}\,\, 21, 1\leq i\leq 2n+1.
\end{align*}
And the comodule structure is given by
\begin{align*}
\rho(m_i)=\left\{\begin{array}{rl}
              x_{11}^{2(s-i)}\chi_{22}^{2i-1}\otimes m_i
    +(-1)^q\lambda\sqrt{\lambda}\omega^{4k(2n+1)(i-1-n)} 
     x_{12}^{2(s-i)}\chi_{21}^{2i-1}\otimes   m_{2n+2-i},
               & i\,\,\text{even},   \vspace{1mm}\\
               x_{11}^{2(s-i)}\chi_{11}^{2i-1}\otimes m_i
               +(-1)^q\lambda\sqrt{\lambda}\omega^{4k(2n+1)(i-1-n)}
               x_{12}^{2(s-i)}\chi_{12}^{2i-1}\otimes m_{2n+2-i},
              & i\,\,\text{odd}.  
              \end{array}
     \right.
\end{align*}

\item Let  $V_{jk}^\prime=\k v_1^\prime\oplus\k v_2^\prime$,  $a=\frac{(-1)^p}{\sqrt{\bar{\mu}}}\omega^{-2k(2n+1)}$, $p\in\Bbb{Z}_2$,  $q\in\Bbb{Z}_2$,  $s\in\overline{1,N}$,
$b=(-1)^q\left(\bar{\mu}\omega^{4k(2n+1)}\right)^{-n}$,   
and denote 
\begin{align*}
w_1&=(v_1^\prime+av_2^\prime)\boxtimes x_{11}^{2s}\chi_{22}^{2n}
        +b\chi_{21}^{2n}\cdot \left[(v_1^\prime+av_2^\prime)
                      \boxtimes x_{12}^{2s+1}\chi_{21}^{2n-1}\right],\\
w_i&= \left\{\begin{array}{rl}
x_{21}\cdot w_{i-1}=\chi_{21}^{i-1}\cdot w_1, 
& i\,\,\text{is even}, 1<i\leq 2n+1,\vspace{1mm}\\
x_{12}\cdot w_{i-1}=\chi_{12}^{i-1}\cdot w_1, 
& i\,\,\text{is odd}, 1<i\leq 2n+1.
\end{array}\right.
\end{align*}
Then $\mathscr{N}_{k,pq}^s=\bigoplus_{i=1}^{2n+1}\k w_i$  is a $2n+1$ dimensional Yetter-Drinfeld 
module over $A_{N\, 2n+1}^{\mu\lambda}$ with the module structure given by 
\begin{align*}
x_{12}\cdot w_i
&=\left\{ \begin{array}{rl}
          \frac1a w_1, & i=1,\\
          w_{i+1},  & i\,\,\,\text{even}\\
          \frac1{a^2} w_{i-1}, & i\,\,\,\text{odd and}\,\,\, i>1
          \end{array}\right.\\
x_{21}\cdot w_i
&=\left\{ \begin{array}{rl}
          \frac{\lambda}{a} w_{2n+1}, & i=2n+1,\vspace{1mm}\\
          \frac1{a^2}w_{i-1},  & i\,\,\,\text{even}\\
          w_{i+1}, & i\,\,\,\text{odd and}\,\,\, i<2n+1,
          \end{array}\right.\\
x_{\alpha\beta}\cdot w_i&=0, \quad \alpha\beta=11\,\,\,\text{or}\,\,\,22,
\end{align*}
and the comodule structure given by 
\[
\rho(w_i)
=\left\{\begin{array}{ll}
x_{11}^{2s}\chi_{22}^{i-1}\chi_{22}^{2n-i+1}\otimes w_i
     +\frac{b}{a^{2(i-1)}}
        x_{12}^{2s}\chi_{21}^{i-1}\chi_{21}^{2n-i+1}\otimes w_{2n-i+2}, & i\,\text{even}\\
x_{11}^{2s}\chi_{22}^{2n-i+1}\chi_{11}^{i-1}\otimes w_i
     +\frac{b}{a^{2(i-1)}}
       x_{12}^{2s}\chi_{21}^{2n-i+1}\chi_{12}^{i-1}\otimes w_{2n-i+2}, & i\,\text{odd}.
\end{array}\right.       
\]

\item $\mathscr{K}_{k,p}^s=\bigoplus_{i=1}^{2n+1}\mathbb{K}w_i$, where 
\begin{align*}
w_1&= v\boxtimes x_{11}^{2s-1}+(-1)^p\sqrt{\lambda}\omega^{-4kn(2n+1)}\chi_{22}^{2n}\cdot\left[v\boxtimes x_{12}^{2s-2}\chi_{21}\right],\\
w_i&= \left\{\begin{array}{rl}
x_{22}\cdot w_{i-1}=\chi_{22}^{i-1}\cdot w_1, 
& i\,\,\text{is even}, 1<i\leq 2n+1,\vspace{1mm}\\
x_{11}\cdot w_{i-1}=\chi_{11}^{i-1}\cdot w_1, 
& i\,\,\text{is odd}, 1<i\leq 2n+1,
\end{array}\right.
\end{align*}
$s\in\overline{1,N}$, $p\in\Bbb{Z}_2$, $V_{k}=\mathbb{K}v$.

\item $\mathscr{M}_{k,p}^s=\bigoplus_{i=1}^{2n+1}\k m_i$, where 
\begin{align*}
m_1&= v\boxtimes x_{11}^{2s}\chi_{22}^{2n}
       +a\chi_{21}^{2n}\cdot \left[v\boxtimes x_{12}^{2s+1}\chi_{21}^{2n-1}\right],\\
m_i&= \left\{\begin{array}{rl}
x_{21}\cdot m_{i-1}=\chi_{21}^{i-1}\cdot m_1, 
& i\,\,\text{is even}, 1<i\leq 2n+1,\vspace{1mm}\\
x_{12}\cdot m_{i-1}=\chi_{12}^{i-1}\cdot m_1, 
& i\,\,\text{is odd}, 1<i\leq 2n+1, 
\end{array}\right.
\end{align*}
$V_{k}^\prime=\k v$, $a=(-1)^p \omega^{-4kn(2n+1)}\tilde{\mu}^{-2n}$, 
$p\in\Bbb{Z}_2$, $s\in\overline{1,N}$.
\end{enumerate}

\providecommand{\bysame}{\leavevmode\hbox to3em{\hrulefill}\thinspace}
\providecommand{\MR}{\relax\ifhmode\unskip\space\fi MR }
\providecommand{\MRhref}[2]{%
  \href{http://www.ams.org/mathscinet-getitem?mr=#1}{#2}
}
\providecommand{\href}[2]{#2}

\end{document}